\pdfoutput=1
\documentclass[pdflatex,sn-mathphys-num]{sn-jnl}

\usepackage{float}           
\usepackage{threeparttable}  
\usepackage[table]{xcolor}   
\usepackage{colortbl}        
\usepackage{multirow}
\usepackage{amsmath}
\usepackage{mathtools}
\usepackage{amsfonts}
\usepackage{amssymb}
\usepackage{amsthm}
\usepackage[title]{appendix}
\usepackage{xcolor}
\usepackage[normalem]{ulem} 
\usepackage{textcomp}
\usepackage{manyfoot}
\usepackage{booktabs}
\usepackage{algorithm}
\usepackage{algorithmicx}
\usepackage{algpseudocode}
\usepackage{listings}
\usepackage{mathrsfs}
\usepackage[utf8]{inputenc}
\usepackage[T1]{fontenc}

\newcommand{\R}{\mathbb{R}}

\newcommand{\J}{\mathcal{J}}

\newcommand{\dsum}{\displaystyle\sum}

\newtheorem{theorem}{Theorem}[section]
\newtheorem{proposition}[theorem]{Proposition}%
\newtheorem{example}{Example}[section]%
\newtheorem{remark}{Remark}[section]%
\newtheorem{lemma}[theorem]{Lemma}
\newtheorem{definition}{Definition}

\newtheorem{A}{Assumption}

\begin{document}


\title[Article Title]{A Partially Derivative-Free Proximal Method for Composite Multiobjective Optimization in the H\"{o}lder Setting}


\author*[1]{\fnm{V.S.} \sur{Amaral}}\email{vitalianoamaral@ufpi.edu.br}

\author[2]{\fnm{P.B.} \sur{Assunção}}\email{pedro.filho@ifg.edu.br}
\equalcont{These authors contributed equally to this work.}

\author[3]{\fnm{D.R.} \sur{Souza}}\email{souzadr@ime.unicamp.br}
\equalcont{These authors contributed equally to this work.}

\affil*[1]{\orgdiv{Department of Mathematics}, \orgname{Federal University of Piauí - UFPI}, \orgaddress{\street{Avenida Petrônio Portela}, \city{Teresina}, \postcode{64049-550}, \state{PI}, \country{Brasil}}}

\affil[2]{\orgdiv{Department of academic areas}, \orgname{Federal Institute of Education, Science and Technology of Goiás - IFG}, \orgaddress{\street{Rua 64 Parque Lago}, \city{Formosa}, \postcode{73813-816}, \state{GO}, \country{Brasil}}}

\affil[3]{\orgdiv{Department of applied mathematics}, \orgname{Institute of Mathematics, Statistics and Scientific Computing - UNICAMP}, \orgaddress{\street{Rua Sérgio Buarque de Holanda}, \city{Campinas}, \postcode{13083-859}, \state{SP}, \country{Brasil}}}


\abstract{This paper presents an algorithm for solving multiobjective optimization problems involving composite functions, where we minimize a quadratic model that approximates $F(x) - F(x^k)$ and that can be derivative-free. We establish theoretical assumptions about the component functions of the composition and provide comprehensive convergence and complexity analysis. 
Specifically,  we prove that, under these assumptions and suitable choices of the parameters, the proposed method finds an $\epsilon$-Pareto critical point in at most $\mathcal{O}\left(\epsilon^{-\frac{\beta+1}{\beta}}\right)$ iterations, where $\beta$ denotes the H\"{o}lder exponent of the gradient.
 The algorithm incorporates gradient approximations and scaling matrices $B_j^k$ to achieve an optimal balance between computational accuracy and efficiency.
Numerical experiments on a collection of benchmark problems are presented, illustrating the practical behavior of the proposed approach and its competitiveness with existing composite algorithms.
}



\maketitle

\section{Introduction}\label{sec1}

Multiobjective optimization is characterized by the simultaneous minimization (or maximization) of several objective functions. In general, these objectives are conflicting, which prevents the existence of a single point capable of optimizing all objective functions at once. Therefore, the concept of Pareto optimality is employed to characterize solutions of multiobjective problems. Problems that deal with multiple objectives simultaneously are common in various fields, including engineering, industrial processes, telecommunications, and game theory, as shown in the works \cite{BhaskarGuptaRay, BeltonStewart2002, Miettinen1998, Stewart2008, Martinez-MunozMartiYepes2025, Rgamemontc}.

There are several iterative approaches for solving multiobjective optimization (MO) problems. To the best of our knowledge, there are three principal approaches: scalarization \cite{Eichfelder2008}, heuristic \cite{laumanns2002combining}, and direct methods, the latter also being considered an extension of scalar-valued methods to multiobjective (vector) optimization. This approach was initiated in \cite{Fliege2000}, where the authors extended the steepest descent method; see also \cite{fliegeanddrummond2009}. Over the past two decades, various direct methods have sought to extend methods such as Newton \cite{chuong2013newton,fliegeanddrummond2009,mauricionewtonvetorial,wang2019extended}, quasi-Newton \cite{ansary,math8040616,Mahdavi-AmiriSadaghiani2020,Morovati2017,POVALEJ2014765,QuGohChan2011,QU2014503, Prudente2022, Prudente2024}, conjugate gradient \cite{Goncalves2019,cgvo}, projected gradient \cite{luis-jef-yun,fazzio2019convergence,fukuda2011convergence,fukuda2013inexact,mauricio&iusem}, proximal methods \cite{bonnel2005proximal,ceng2010hybrid,ceng2007approximate,chuong2011generalized,chuong2011hybrid, fukuda2019, fukuda2023, Nonmonotone-Peng2025, CHEN2025116422, Bello-Cruz2025}, (generalized) conditional gradient \cite{Assuncao2021, Assuncao2023, Goncalves2025}, and Newton-type proximal gradient \cite{Ansary04052023}.

In many practical optimization problems, direct access to the derivatives of the functions involved is often unavailable, which makes their computation challenging. Additionally, evaluating the objective function and its derivatives can be computationally expensive. As a result, there is growing interest in methods that offer low computational cost while maintaining good worst-case complexity in terms of function evaluations and iterations. In this context, research on derivative-free methods has advanced significantly for scalar problems.  For further details, see \cite{Amaral2025, Grapiglia2023, Grapiglia2024} and the references therein.

The objective of this work is to develop a partially derivative-free method for solving the nonsmooth multiobjective optimization problem
\begin{equation}\label{P1}
\min_{x\in\mathbb{R}^n} F(x),
\end{equation}
where $F:\mathbb{R}^n\to{(\mathbb{R}\cup\{+\infty\})^m}$ is given by $F(x)=(F_1(x),\ldots,F_m(x))$, with
\begin{equation*}\label{composite}
F_j\coloneqq f_j+h_j,\qquad j\in\mathcal{J}\coloneqq\{1,\ldots,m\},
\end{equation*}
where $f_j:\mathbb{R}^n\to\mathbb{R}$ is differentiable and
$$
h_j=\widetilde{h}_j+\delta_{\mathcal{C}},\qquad j\in\mathcal{J}.
$$
Here each $\widetilde{h}_j:\mathbb{R}^n\to\mathbb{R}$ is convex and finite-valued, $\mathcal{C}\subseteq\mathbb{R}^n$ is a nonempty closed convex set common to all components, and $\delta_{\mathcal{C}}:\R^n\to\mathbb{R}\cup\{+\infty\}$ is its indicator function, defined by
$$
\delta_{\mathcal{C}}(x)=
\begin{cases}
0,&x\in \mathcal{C},\\
+\infty,&x\notin \mathcal{C}.
\end{cases}
$$
Consequently, each $h_j$ is proper, lower semicontinuous, and convex, but not necessarily differentiable. The effective domain of $F$ is
$$
\operatorname{dom}F\coloneqq\left\{x\in\mathbb{R}^n:F_j(x)<+\infty\text{ for every }j\in\mathcal{J}\right\}=\mathcal{C}.
$$
This formulation includes unconstrained problems when $\mathcal{C}=\mathbb{R}^n$ and incorporates convex constraints through $\delta_{\mathcal{C}}$.

The problem of minimizing a function expressed as the sum of two functions has several applications; see \cite[Section 5.2]{fukuda2019}. If $\widetilde{h}_j\equiv 0$ for all $j\in\mathcal{J}$, then $h_j$ is the indicator function of $\mathcal{C}$. Thus, the problem in \eqref{P1} is equivalent to solving the following problem:
\[\min_{x \in \mathcal{C}}    (f_1(x), \ldots, f_m(x)).\]
In addition, the problem in \eqref{P1} can be applied to robust optimization problems, which involve uncertainties in their parameters and essentially consist of solving the problem in the worst-case scenario; see \cite{Chen2019, fukuda2019} and their references.

Recently, composite problems such as in \eqref{P1} have been addressed using various methods, including proximal gradient \cite{fukuda2019, Nonmonotone-Peng2025, Bello-Cruz2025, CHEN2025116422}, generalized conditional gradient \cite{Assuncao2023}, and Newton-type proximal gradient  \cite{Ansary04052023}. In all references, the  component functions $f_j : \mathbb{R}^n \to \mathbb{R}$ for $j \in \mathcal{J}$ are continuously differentiable.  
In \cite{fukuda2023, Assuncao2023}, the authors further assumed that $f_j$ have Lipschitz continuous gradients, while the functions $h_j : \mathbb{R}^n \to \mathbb{R} \cup \{+\infty\}$ are proper and convex. Notably, under the assumption that $h_j$ are lower semicontinuous, Assunção et al. \cite{Assuncao2023} proposed a generalized version of the conditional gradient method (Frank-Wolfe) for multiobjective composite optimization.  
In \cite{Ansary04052023, Bello-Cruz2025}, the functions $f_j$ are assumed to be convex, and $h_j$ are considered continuous and convex; Bello-Cruz \cite{Bello-Cruz2025} additionally assumed that $h_j$ are proper. In contrast, \cite{CHEN2025116422, Nonmonotone-Peng2025} addressed cases where $f_j$ are continuously differentiable but not necessarily convex, while $h_j$ are convex (and proper, lower semicontinuous in \cite{CHEN2025116422}).

Given an iteration $x^k\in\operatorname{dom}F$, our approach finds $x^{k+1}$ such that $F_j(x^{k+1})$ is sufficiently smaller than $F_j(x^k)$ for all $j \in \mathcal{J}$. As discussed by Calderón et al.~\cite{Lizet}, this can be done by defining an approximate model for each $F_j(x) - F_j(x^k)$ around $x^k$ and approximately minimizing the maximum of these models. In our case, to approximate
	\[\Phi_{x^k}(x) + \frac{\sigma_k}{2}\|x-x^k\|^2\]
where 
\begin{equation}\label{Phi}
    \Phi_{x^k}(x) \coloneqq \max_{j \in \mathcal{J}} \big[ \langle g_{f_j}(x^k,\lambda_k) + \tfrac{1}{2} B^k_j(x-x^k), \, x-x^k \rangle + h_j(x) - h_j(x^k) \big],
\end{equation}
$g_{f_j}(x^k,\lambda_k)$ is an approximation of $\nabla f_j(x^k)$ and $B_j^k$ is a symmetric positive semidefinite matrix. To our knowledge, Fliege and Svaiter \cite{Fliege2000} were the first to use this approach to solve multiobjective optimization problems.

One of the novelties of our approach lies in the extension of the derivative-free method described by Grapiglia in \cite{Grapiglia2024} to multiobjective optimization. To the best of our knowledge, our work is one of the pioneers in applying this type of method to multiobjective optimization. In \cite{Grapiglia2024}, Grapiglia proposed a quadratic regularization method to minimize a function \( f:\mathbb{R}^n \to \mathbb{R} \), where he used finite-difference approximations for the gradients and proved that, when the objective function is bounded below and has Lipschitz gradients, the method requires at most \(\mathcal{O}(\epsilon^{-2})\) iterations to generate an approximate stationary point with \(\epsilon\) accuracy.

In our work, we assume a H\"{o}lder-Lipschitz condition on the gradient of $f_j$ and establish a worst-case iteration complexity bound of $\mathcal{O}\left(\epsilon^{-\frac{\beta+1}{\beta}}\right)$ for finding an $\epsilon$-Pareto critical point (Definition~\ref{epspareto}) of problem~\eqref{P1}, under suitable regularity and parameter conditions, where $\beta=\min\{\beta_j:j\in\mathcal{J}\}$ and $\beta_j$ denotes the H\"{o}lder exponent of $\nabla f_j$. Our complexity bound is order-wise consistent with the bound established by Calderón et al.~\cite{Lizet} for the first-order version of their method of order $p$, as well as with the complexity order obtained by Grapiglia et al.~\cite{PinheiroGrapiglia2024}. Furthermore, when $\beta_1 = \dots = \beta_m = 1$, i.e., when only Lipschitz conditions are assumed for the gradients of $f_j$, our bound becomes  
$
\mathcal{O}\left(\epsilon^{-2}\right),
$ which is consistent with previously established bounds; see, for example, Grapiglia et al.~\cite{Grapiglia2015}. In the case $m = 1$, our complexity order coincides with that obtained by Martínez~\cite{Martinez2017} for the first-order version of their method of order $p$. Finally, we report numerical experiments on robust problems: two biobjective problems with H\"{o}lder continuous gradients, and a comparison with ProxGrad and CondG on $71$ benchmark problems under a common stationarity measure.

The remainder of the article is structured as follows. First, in Section \ref{preli}, we present the notation and important results for understanding the work. Next, in Section \ref{Sec:LocalAnalysis}, we provide a detailed description of the proposed proximal point method. In Section \ref{seccomplex}, we analyze the iteration complexity of the method in the worst-case scenario. To further demonstrate the applicability of our approach, we provide numerical experiments in Section \ref{sec5}. Finally, in Section \ref{conlusions}, we present the concluding remarks of the paper.

\section{Notations and Auxiliary Results}\label{preli}

We now introduce some notations, definitions, and results that will be used throughout the paper. We denote by $ \mathbb{N} \coloneqq \{0, 1, 2, \ldots\} $ the set of nonnegative integers, and by $ \mathbb{N}^* \coloneqq \{1, 2, 3, \ldots\} $ the set of positive integers. The sets $ \mathbb{R} $, 
 $\mathbb{R}_+ $, and $ \mathbb{R}_{++} $ represent the set of real numbers, the set of nonnegative real numbers, and the set of positive real numbers, respectively. Let $ \mathcal{J} = \{1, \ldots, m\} $, we define:
\[\mathbb{R}^m_+ \coloneqq \{ u \in \mathbb{R}^m \mid u_j \geq 0,\ \forall j \in \mathcal{J} \}
\quad \text{and} \quad
\mathbb{R}^m_{++} \coloneqq \{ u \in \mathbb{R}^m \mid u_j > 0,\ \forall j \in \mathcal{J} \}.\]
If $u, v \in {\mathbb R}^{m}$, then $v \leq u$ means that $u-v \in {\mathbb R}^{m}_{+}$, and if $v< u$ then $u-v \in {\mathbb R}^{m}_{++}$. 
The symbol $\langle \cdot, \cdot \rangle$ is the usual inner product in $\R^n$ and $\| \cdot \|$ denotes the Euclidean norm. 

For a symmetric matrix $A \in \mathbb{R}^{n\times n}$, we write $A \succeq 0$ or $0 \preceq A$ to indicate that $A$ is positive semidefinite, that is, $x^\top A x \geq 0$ for all $x \in \mathbb{R}^n$.
\begin{definition}
	We say that a point $x^* \in {\operatorname{dom}F}$ is a Pareto optimal point for the problem \eqref{P1} if there does not exist $x \in {\operatorname{dom}F}$ such that
	$$F(x) \leq F(x^*)\mbox{ and } F(x) \neq F(x^*).$$
	We say that $x^* \in \operatorname{dom}F$ is a weakly Pareto optimal point for the problem~\eqref{P1} if there does not exist $x \in \operatorname{dom}F$ such that $F(x) < F(x^*)$.
\end{definition}

Conditions of Fritz John type are classical for finite-valued, locally Lipschitz objectives; see \cite[Section 3.2.1 (First-Order Conditions)]{Miettinen1998}. For problem~\eqref{P1}, the following condition is established in Proposition~\ref{prop:stationary}. If $x^*$ is Pareto optimal for the problem \eqref{P1}, then there exists $\gamma\in\mathbb{R}^m_{+}$ with $\sum_{j=1}^m\gamma_j=1$ such that
\begin{equation}\label{stationary}
	0\in \sum_{j=1}^m\gamma_j (\nabla f_j(x^*)+\partial  h_j(x^*)),
\end{equation}
where $\partial h_j$ is the subdifferential of the convex function $h_j:\mathbb{R}^n\to \mathbb{R}{\cup\{+\infty\}}$ defined, for $\bar{x}\in\mathcal{C}$, by
\begin{equation*}\label{epsilonsub}
	\partial h_j(\bar{x})=\{y\in\mathbb{R}^n~:~ h_j(z)\geq h_j(\bar{x})+\langle y, z-\bar{x}\rangle, \quad \forall z\in\mathbb{R}^n\}.
\end{equation*}

With this in mind, we introduce the following definition, which extends \cite[Definition~1]{Goncalves2026Cubic} to problem~\eqref{P1}.

\begin{definition}\label{epspareto}
Consider problem~\eqref{P1}.
\begin{enumerate}
    \item A point $x^*\in\operatorname{dom}F$ is called a \emph{Pareto critical point} (also called Pareto stationary \cite{fukuda2019}) if there
    exists $\gamma\in\mathbb{R}^m_+$ with
    $\sum_{j=1}^m\gamma_j=1$ such that
    \begin{equation*}\label{eq:pareto-critical}
        0\in\sum_{j=1}^m\gamma_j
        \bigl(\nabla f_j(x^*)+\partial h_j(x^*)\bigr).
    \end{equation*}

   \item Given $\epsilon \geq 0$, a point $x^* \in \operatorname{dom} F$ is called an
$\epsilon$-Pareto critical point if there exists $\gamma \in \mathbb{R}^m_+$ with
$\sum_{j=1}^m \gamma_j = 1$ such that
\[
d\Big(0, \sum_{j=1}^m \gamma_j\big(\nabla f_j(x^*) + \partial h_j(x^*)\big)\Big) \leq {\epsilon},
\]
where $d(0,S) \coloneqq \inf_{s \in S} \|s\|$ denotes the distance from the origin to a set $S \subseteq \mathbb{R}^n$.
    
\end{enumerate}
\end{definition}

When $\widetilde{h}_j\equiv 0$ for all $j\in\mathcal{J}$ and $\mathcal{C}=\mathbb{R}^n$, item~2 of Definition~\ref{epspareto} reduces to the $\epsilon$-approximate weakly Pareto critical point of \cite[Definition~1]{Goncalves2026Cubic}. Moreover, for $x\in\mathcal{C}$, we have $\partial h_j(x)=\partial\widetilde{h}_j(x)+N_{\mathcal{C}}(x)$ \cite[Theorem~23.8]{Rockafellar1970}, so that, for $\gamma\in\mathbb{R}^m_+$ with $\sum_{j=1}^m\gamma_j=1$, the set $\sum_{j=1}^m\gamma_j\bigl(\nabla f_j(x)+\partial h_j(x)\bigr)=\sum_{j=1}^m\gamma_j\bigl(\nabla f_j(x)+\partial\widetilde{h}_j(x)\bigr)+N_{\mathcal{C}}(x)$ is the sum of a compact set and a closed cone, and hence is closed. Therefore, the inequality in item~2 with $\epsilon=0$ is equivalent to the inclusion in item~1.

It is important to note that $\partial {\delta_{\mathcal{C}}}(x)=N_{\mathcal{C}}(x)$, for all $x\in\mathcal{C}$, where $N_{\mathcal{C}}(x)$ is the normal cone to $\mathcal{C}$ at $x$, defined by
$$
N_{\mathcal{C}}(x)\coloneqq\{z\in \mathbb{R}^n~:~ \langle z, y-x\rangle \leq 0, \quad \forall\, y\in \mathcal{C}\}.
$$

The following lemma, a min--max form of the first-order optimality condition for convex functions on $\mathcal{C}$, is used to prove Proposition~\ref{prop:stationary}, which shows that $\partial h_j(x)\neq\emptyset$ for every $x\in\mathcal{C}$ and that every Pareto optimal point of problem~\eqref{P1} satisfies~\eqref{stationary}.

\begin{lemma}\label{lem:minmax}
Let $a_j:\mathbb{R}^n\to\mathbb{R}$, $j\in\mathcal{J}$, be convex and let $\bar{x}\in\mathcal{C}$ minimize $\max_{j\in\mathcal{J}}a_j$ over $\mathcal{C}$. Then there exist $\gamma\in\mathbb{R}^m_+$ with $\sum_{j=1}^m\gamma_j=1$, $v_j\in\partial a_j(\bar{x})$ for $j\in\mathcal{J}$, and $\omega\in N_{\mathcal{C}}(\bar{x})$ such that
$$
0=\sum_{j=1}^m\gamma_j v_j+\omega.
$$
\end{lemma}

\begin{proof}
Let $\Delta\coloneqq\{\gamma\in\mathbb{R}^m_+:\sum_{j=1}^m\gamma_j=1\}$ and $K(\gamma,y)=\sum_{j=1}^m\gamma_j a_j(y)$. For each $y$, the maximum of $K(\cdot,y)$ over $\Delta$ equals $\max_j a_j(y)$. The function $K$ is continuous, concave in $\gamma$ and convex in $y$; moreover, $\Delta$ is compact and $\mathcal C$ is closed and convex. Thus \cite[Corollary 37.3.2]{Rockafellar1970} gives
$$
\min_{y\in\mathcal{C}}\max_{j\in\mathcal{J}}a_j(y)=\sup_{\gamma\in\Delta}\inf_{y\in\mathcal{C}}\sum_{j=1}^m\gamma_ja_j(y).
$$
The function $\gamma\mapsto\inf_{y\in\mathcal{C}}\sum_{j=1}^m\gamma_ja_j(y)$ is upper semicontinuous as the infimum of continuous functions. Since its supremum is finite by the equality above and $\Delta$ is compact, it is attained at some $\gamma\in\Delta$. Hence
$$
\inf_{y\in\mathcal{C}}\sum_{j=1}^m\gamma_ja_j(y)\leq\sum_{j=1}^m\gamma_ja_j(\bar{x})\leq\max_{j\in\mathcal{J}}a_j(\bar{x})=\inf_{y\in\mathcal{C}}\sum_{j=1}^m\gamma_ja_j(y),
$$
that is, $\bar{x}$ minimizes the finite convex function $\sum_{j=1}^m\gamma_ja_j$ over $\mathcal{C}$. By \cite[Theorems 27.4 and 23.8]{Rockafellar1970}, there exist $v_j\in\partial a_j(\bar{x})$ for the indices with $\gamma_j>0$ and $\omega\in N_{\mathcal{C}}(\bar{x})$ such that $0=\sum_{j=1}^m\gamma_jv_j+\omega$. For the indices with $\gamma_j=0$, any $v_j\in\partial a_j(\bar{x})$ can be taken, since this set is nonempty by \cite[Theorem 23.4]{Rockafellar1970}.
\end{proof}

\begin{proposition}\label{prop:stationary}
For every $x\in\mathcal{C}$ and $j\in\mathcal{J}$, one has $\partial\widetilde{h}_j(x)+N_{\mathcal{C}}(x)\subseteq\partial h_j(x)$; in particular, $\partial h_j(x)\neq\emptyset$. Moreover, if $x^*$ is Pareto optimal for the problem \eqref{P1}, then \eqref{stationary} holds.
\end{proposition}

\begin{proof}
Let $x\in\mathcal{C}$, $v\in\partial\widetilde{h}_j(x)$ and $\omega\in N_{\mathcal{C}}(x)$. For every $z\in\mathcal{C}$,
$$
h_j(z)=\widetilde{h}_j(z)\geq\widetilde{h}_j(x)+\langle v,z-x\rangle\geq h_j(x)+\langle v+\omega,z-x\rangle,
$$
and the inequality is trivial for $z\notin\mathcal{C}$. Hence $v+\omega\in\partial h_j(x)$. The nonemptiness follows from $\partial\widetilde{h}_j(x)\neq\emptyset$ \cite[Theorem 23.4]{Rockafellar1970} and $0\in N_{\mathcal{C}}(x)$.

Now let $x^*$ be Pareto optimal and define
$$
a_j(y)\coloneqq\langle\nabla f_j(x^*),y-x^*\rangle+\widetilde{h}_j(y)-\widetilde{h}_j(x^*),\quad j\in\mathcal{J}.
$$
Suppose that there exists $y\in\mathcal{C}$ with $a_j(y)<0$ for all $j\in\mathcal{J}$, and let $y_t\coloneqq x^*+t(y-x^*)\in\mathcal{C}$ for $t\in(0,1]$. By the differentiability of $f_j$ and the convexity of $\widetilde{h}_j$,
$$
F_j(y_t)-F_j(x^*)\leq t\,a_j(y)+o(t),\quad j\in\mathcal{J},
$$
so $F_j(y_t)<F_j(x^*)$ for all $j\in\mathcal{J}$ when $t$ is sufficiently small, which contradicts the Pareto optimality of $x^*$. Therefore, $\max_{j\in\mathcal{J}}a_j(y)\geq0=\max_{j\in\mathcal{J}}a_j(x^*)$ for all $y\in\mathcal{C}$. Applying Lemma~\ref{lem:minmax} and noting that $\partial a_j(x^*)=\nabla f_j(x^*)+\partial\widetilde{h}_j(x^*)$, we obtain $\gamma\in\mathbb{R}^m_+$ with $\sum_{j=1}^m\gamma_j=1$, $v_j\in\partial\widetilde{h}_j(x^*)$ and $\omega\in N_{\mathcal{C}}(x^*)$ such that
$$
0=\sum_{j=1}^m\gamma_j\big(\nabla f_j(x^*)+v_j\big)+\omega=\sum_{j=1}^m\gamma_j\big(\nabla f_j(x^*)+v_j+\omega\big).
$$
Since $v_j+\omega\in\partial h_j(x^*)$, \eqref{stationary} follows.
\end{proof}

Let $\varphi : \mathbb{R}^n \to \mathbb{R}$ be a differentiable function. We say that $\nabla \varphi$ is H\"{o}lder continuous with constant $L > 0$ if there exists ${\nu} \in (0,1]$ such that
$$
\| \nabla \varphi(y) - \nabla \varphi(x) \| \leq L \| y - x \|^{{\nu}}, \quad \forall x, y \in \mathbb{R}^n.
$$
If ${\nu} = 1$, we say that $\nabla \varphi$ is Lipschitz continuous with constant $L > 0$.

\section{Proposed Method}\label{Sec:LocalAnalysis}

For the development of the method proposed in this work, for each $j \in \mathcal{J}$, we consider the function 
\begin{equation}\label{gfj}
	g_{f_j}: \mathbb{R}^n \times [0,1] \to \mathbb{R}^n \quad \text{where} \quad \lim_{\lambda \to 0} g_{f_j}(x, \lambda) = \nabla f_j(x)\mbox{ and }g_{f_j}(x, 0) = \nabla f_j(x).
\end{equation}
We observe that the function in \eqref{gfj} can be defined as $g_{f_j} = \nabla f_j$. However, when the calculation of the gradient of $f_j$ is prohibitive, this definition is not recommended. In this case, it is preferable to define $g_{f_j}$ as an approximation of the gradient $\nabla f_j$ that is less expensive to obtain. Three possible definitions of $g_{f_j}$ are presented in the following remark.

\begin{remark}\label{diff}
If $f_j$ is differentiable and $e_i,\,\,i=1,\ldots,n$ denotes the $i$-th canonical vector of $\R^n$, then,  when $\lambda\neq 0$, $g_{f_j}$ can be defined in one of the following ways:
\begin{itemize}
\item[(i)] Forward difference: $$g_{f_j}(x,\lambda)=	
\bigg[\dfrac{ f_j(x+\lambda e_1)- f_j(x)}{\lambda},\cdots,\dfrac{ f_j(x+\lambda e_n)- f_j(x)}{\lambda}\bigg];$$
\item[(ii)] Backward difference: $$g_{f_j}(x,\lambda)=	
\bigg[\dfrac{ f_j(x)-f_j(x-\lambda e_1)}{\lambda},\cdots,\dfrac{f_j(x)- f_j(x-\lambda e_n)}{\lambda}\bigg];$$
\item[(iii)] Central difference: $$g_{f_j}(x,\lambda)=	
\bigg[\dfrac{ f_j(x+\lambda e_1)-f_j(x-\lambda e_1)}{2\lambda},\cdots,\dfrac{f_j(x+\lambda e_n)- f_j(x-\lambda e_n)}{2\lambda}\bigg].$$
    \end{itemize}

\end{remark}

It is important to note that, since $f_j$ is differentiable at $x$, each of the approximations defined above converges to the gradient $\nabla f_j(x)$ as $\lambda \to 0$. More precisely, we have
\(
\lim_{\lambda \to 0} g_{f_j}(x, \lambda) = \nabla f_j(x),
\)
in all three cases (forward, backward, or central differences).

\subsection{The Method}\label{sec.01cap2}
 In this subsection, we present a partially derivative-free method, which can be made fully derivative-free, for solving the problem  \eqref{P1}.

\begin{algorithm}[!ht]
	\caption{\textsc{ Partially Derivative-Free Proximal Method - PDFPM}}\label{algfree}
	\medskip
	
Let $x^0\in{\operatorname{dom}F}$, $\alpha, \epsilon \in(0,1)$, $\sigma_0\geq 1$ , and $0\preceq B^0_j\in\R^{n\times n}$ for all $j \in \J$. Initialize $k\leftarrow 0$.
\begin{description}
	\item[ \textbf{Step 1.}]  For 
		\begin{equation*}\label{d1}
		0\leq\lambda_k\leq \dfrac{\epsilon}{\sigma_k\sqrt{n}}
	\end{equation*}
	and each $j\in\mathcal{J}$, compute $g_{f_j}(x^k,\lambda_k)$ (see definition in \eqref{gfj}).
	
	\item[ \textbf{Step 2.}]  Find $\bar{x}^k\in\R^n$ solution of the following subproblem: 
    \begin{equation}\label{probminmax}
		\min_{x\in\R^n}\Phi_{x^k}(x)+\frac{\sigma_k}{2}\|x-x^k\|^2,
	\end{equation} where $\Phi_{x^k}(x)$ is defined in \eqref{Phi}.
	
	\item[ \textbf{Step 3.}] If $\sigma_k\|\bar{x}^k-x^k\|\geq \frac{\epsilon}{3\big(1+\max_{j\in\mathcal{J}}  \operatorname{tr}(B_j^k)\big)}$, go to \textit{\textbf{Step~4}}. Otherwise, STOP and declare $\bar{x}^{k}$ an acceptable solution. 
	
	\item[ \textbf{Step 4.}] If	
	\begin{equation}\label{eq.te}
		F_j(\bar{x}^k)\leq F_j(x^k)-\dfrac{\alpha}{2}\sigma_k\|\bar{x}^{k}-x^k\|^2,\quad \forall~j\in\mathcal{J},			
	\end{equation}
	holds, set $x^{k+1} \coloneqq \bar{x}^k$, $\sigma_{k+1}\coloneqq\sigma_k$ and choose a matrix $0 \preceq B_j^{k+1} \in \R^{n\times n}$, $k \leftarrow k + 1$, and go to \textbf{\textit{Step~1}}. Otherwise, define $\sigma_k\leftarrow 2\sigma_k$ and go to \textbf{\textit{Step~1}}.
\end{description}
\end{algorithm}

The parameter $\alpha$ controls the reduction level of $F_j$ in \eqref{eq.te}, where values close to $1$ result in a more aggressive reduction, which may cause the penalty parameter $\sigma_k$ to grow unnecessarily. Therefore, it is crucial to tune $\alpha$ carefully. The descent test in \eqref{eq.te} shows that small values of $\sigma_k$ tend to produce larger steps, which can speed up the convergence of the method.  As will be shown in Proposition~\ref{prop:epspareto}, the criterion established in Step~3 of Algorithm~\ref{algfree} {yields an $\epsilon$-Pareto critical point under conditions on the problem constants and on the algorithmic parameters}.

The remainder of this subsection is devoted to establishing the well-definedness of the steps of Algorithm~\ref{algfree}.
Assume $x^k\in\operatorname{dom}F$. Each $h_j$ is then proper,
since $h_j(x^k)<+\infty$. To ensure the correct formulation of Step~2, we consider the following function:
$$\Psi_{\sigma_k}^{x^k}(x)=\Phi_{x^k}(x)+\dfrac{\sigma_k}{2}\|x-x^k\|^2$$
with $B^k_j$ symmetric positive semidefinite matrices. Since $h_j$ is convex and proper, then there exists a affine function $\overline{h}_j(x)=\langle a,x\rangle+b,\,a,x\in\R^n,\,b\in\R$ such that $h_j(x)\geq \overline{h}_j(x)$ for all $x\in\R^n$, hence we have that $$\displaystyle\lim_{\|x\|\to +\infty}\Psi_{\sigma_k}^{x^k}(x)=+\infty,$$ in other words, $\Psi_{\sigma_k}^{x^k}(x)$ is coercive. Furthermore, it is easy to see that the function $\Psi_{\sigma_k}^{x^k}$ is strongly convex  and lower semicontinuous, the latter because each $h_j$ is lower semicontinuous, and, as a consequence, the problem \eqref{probminmax} has a unique solution, ensuring the well-definedness of Step~2 in Algorithm~\ref{algfree}.

Moreover,
$\Psi_{\sigma_k}^{x^k}(\bar x^k)\le\Psi_{\sigma_k}^{x^k}(x^k)=0$
implies $\bar x^k\in\operatorname{dom}F$, because the subproblem objective
is $+\infty$ outside $\operatorname{dom}F$.
Starting from $x^0\in\operatorname{dom}F$, all iterates and trial
minimizers therefore belong to $\operatorname{dom}F$.

It is important to note that the matrix $B_j^k$ used in Step~2 of Algorithm~\ref{algfree} does not need to be an approximation of the Hessian $\nabla^2 f_j(x^k)$, as the expression in \eqref{Phi} might suggest. It can even be defined as the null matrix, relying only on first-order information, which is useful when the Hessian of $f_j$ is expensive to compute. However, second-order information, when available, can improve the efficiency of the step. The flexibility in choosing $B_j^k$ allows for the use of the Hessian, a low-cost approximation (such as a quasi-Newton type), or the zero matrix, provided that the matrix remains positive semi-definite.

Next, we analyze the well-definedness of Step~4. For this purpose, we consider the following assumptions.

\begin{A}\label{a2} For each $j\in \mathcal{J}$, there exist $L_j, M_j\in\R_{++}$ and $\beta_j\in (0,1]$ such that
\begin{equation}\label{L-holder}
   \|\nabla f_j(y)-\nabla f_j(x)\|\leq L_j\|y-x\|+M_j\|y-x\|^{\beta_j}, \quad \forall\,x,y\in\R^n 
\end{equation}
and 
	\begin{equation}\label{sup.03}
		f_j(y)\leq f_j(x)+\langle\nabla f_j(x),y-x\rangle+ \dfrac{L_j}{2} \|y-x\|^2+\dfrac{M_j}{\beta_j+1} \|y-x\|^{\beta_j+1}, \quad \forall\,x,y\in\R^n.
	\end{equation}		
\end{A}

In Remark \ref{f=s+r}, we present a specific situation in which Assumption~\ref{a2} is satisfied, thereby illustrating a concrete scenario in which this assumption is applicable.
\begin{remark}\label{f=s+r}
	 If $f_j = s_j + r_j$, where $s_j$ has a gradient that is $L_j$-Lipschitz and $r_j$ has a gradient that is $M_j$-H\"{o}lder continuous with exponent $\beta_j$, then \eqref{L-holder} and \eqref{sup.03} hold. In fact, by hypothesis, we have
\begin{equation*}
   \|\nabla s_j(y)-\nabla s_j(x)\|\leq L_j\|y-x\|\mbox{ and } \|\nabla r_j(y)-\nabla r_j(x)\|\leq M_j\|y-x\|^{\beta_j}, \quad \forall\,x,y\in\R^n. 
\end{equation*}
Using $\nabla f_j=\nabla s_j+\nabla r_j$ and the triangle inequality, we obtain  
\begin{eqnarray*}
   \|\nabla f_j(y)-\nabla f_j(x)\|&\leq & \|\nabla s_j(y)-\nabla s_j(x)\|+ \|\nabla r_j(y)-\nabla r_j(x)\|\\
   &\leq &L_j\|y-x\|+ M_j\|y-x\|^{\beta_j}, \quad \forall\,x,y\in\R^n. 
\end{eqnarray*} 
	 	On the other hand, by the hypotheses about $s_j$ and $r_j$ and \cite[Lemma 1]{Yashtini2016} we obtain
\begin{equation*}
	 	s_j(y)\leq s_j(x)+\langle\nabla s_j(x),y-x\rangle+ \dfrac{L_j}{2} \|y-x\|^2, \quad \forall\,x,y\in\R^n,\,\,j\in \mathcal{J},
	 \end{equation*}
	 and 
	 	\begin{equation*}
	 	r_j(y)\leq r_j(x)+\langle\nabla r_j(x),y-x\rangle+\dfrac{M_j}{\beta_j+1} \|y-x\|^{\beta_j+1}, \quad \forall\,x,y\in\R^n,\,\,j\in \mathcal{J}, 
	 \end{equation*}
	 which directly implies \eqref{sup.03}.
\end{remark}

\begin{A}\label{sup.003} For $L_j$ and $M_j$ given in Assumption~\ref{a2}, the following inequality holds for each $j \in \mathcal{J}$: 
	\begin{equation}\label{Lr}
		\|\nabla f_j(x)-g_{f_j}(x,\lambda)\| \leq \lambda\dfrac{\sqrt{n}L_j}{2}+\sqrt{n}\dfrac{M_j}{\beta_j+1}\lambda^{\beta_j}.
	\end{equation}
\end{A}

In the following, we present situations that illustrate how Assumptions~\ref{a2} and~\ref{sup.003} are not restrictive and can be satisfied by a broad class of functions.

\begin{remark}
It is important to highlight that, in Step~1, any function $g_{f_j}$ satisfying \eqref{gfj} and Assumption~\ref{sup.003} may be used; it does not need to be one of the specific constructions presented in Remark~\ref{diff}.
\end{remark}

In the following Lemma \ref{condLH}, we will demonstrate that, if $f_j$ satisfies the condition presented in Remark \ref{f=s+r} and if we consider $g_{f_j}$ as one of the three definitions presented in Remark \ref{diff}, then Assumption \ref{sup.003} holds.
  
\begin{lemma}\label{condLH}  
 Suppose $f_j = s_j + r_j$, where $s_j$ has a gradient that is $L_j$-Lipschitz continuous and $r_j$ has a gradient that is $M_j$-H\"{o}lder continuous with exponent $\beta_j$. Then the condition \eqref{Lr} is satisfied when considering $g_{f_j}(x,\lambda)$ in any of the definitions given in Remark \ref{diff}.	
\end{lemma}

\begin{proof} Considering $s_j$ with a gradient that is $L_j$-Lipschitz continuous and $r_j$ with a gradient that is $M_j$-H\"{o}lder continuous with exponent $\beta_j$, by \cite[Lemma 1]{Yashtini2016} we can conclude that
\begin{equation}\label{eqLj}
		|s_j(y)-s_j(x)-\nabla s_j(x)^\top(y-x)|\leq \frac{L_j}{2}\|y-x\|^2
\end{equation}
	and 
\begin{equation}\label{eqMj}
		|r_j(y)-r_j(x)-\nabla r_j(x)^\top(y-x)|\leq \frac{M_j}{1+\beta_j}\|y-x\|^{\beta_j+1}.
\end{equation}
First, consider the forward finite-difference
approximation. Setting $y=x+\lambda e_i$ in
\eqref{eqLj} and \eqref{eqMj}, we obtain
	\begin{eqnarray*}
		\bigg|\dfrac{s_j(x+\lambda e_i)-s_j(x)}{\lambda}-\dfrac{\partial s_j(x)}{\partial x_i}\bigg|&\leq &\frac{L_j}{2}\lambda
	\end{eqnarray*}
	and 
	\begin{eqnarray*}
		\bigg|\dfrac{r_j(x+\lambda e_i)-r_j(x)}{\lambda}-\dfrac{\partial r_j(x)}{\partial x_i}\bigg|&\leq &\frac{M_j}{1+\beta_j}\lambda^{\beta_j},
	\end{eqnarray*}
	this implies that
	\begin{equation*}
        \|\nabla s_j(x)-g_{s_j}(x,\lambda)\|\le \frac{\sqrt{n}\,L_j}{2}\,\lambda \;\text{and}\;
        \|\nabla r_j(x)-g_{r_j}(x,\lambda)\|\le \frac{\sqrt{n}\,M_j}{\beta_j+1}\,\lambda^{\beta_j}.
        \end{equation*}
        Next, consider the backward finite-difference
approximation. Setting $y=x-\lambda e_i$ in
\eqref{eqLj} and \eqref{eqMj}, we obtain	
		\begin{eqnarray*}
			\bigg|\dfrac{s_j(x)-s_j(x-\lambda e_i)}{\lambda}-\dfrac{\partial s_j(x)}{\partial x_i}\bigg|&\leq &\frac{L_j}{2}\lambda
		\end{eqnarray*}
		and 
		\begin{eqnarray*}
			\bigg|\dfrac{r_j(x)-r_j(x-\lambda e_i)}{\lambda}-\dfrac{\partial r_j(x)}{\partial x_i}\bigg|&\leq &\frac{M_j}{1+\beta_j}\lambda^{\beta_j},
		\end{eqnarray*}
		this implies that
		$$	\|\nabla s_j(x)-g_{s_j}(x,\lambda)\|\leq\dfrac{\sqrt{n}L_j}{2}\lambda \text{ and } 	\|\nabla r_j(x)-g_{r_j}(x,\lambda)\|\leq\dfrac{\sqrt{n}M_j}{\beta_j+1}\lambda^{\beta_j}.$$
		Finally, consider the central finite-difference
approximation. Using its definition in
Remark~\ref{diff} and the preceding estimates,
we obtain
		\begin{eqnarray*}
			\bigg|\dfrac{s_j(x+\lambda e_i)-s_j(x-\lambda e_i)}{2\lambda}-\dfrac{\partial s_j(x)}{\partial x_i}\bigg|
			&\leq&\tfrac{1}{2}\bigg|\dfrac{s_j(x+\lambda e_i)-s_j(x)}{\lambda}-\dfrac{\partial s_j(x)}{\partial x_i}\bigg|\\
            &+&\tfrac{1}{2}\bigg|\dfrac{s_j(x)-s_j(x-\lambda e_i)}{\lambda}-\dfrac{\partial s_j(x)}{\partial x_i}\bigg|\\
			&\leq&\tfrac{1}{2}\bigg(\dfrac{L_j}{2}\lambda+\dfrac{L_j}{2}\lambda\bigg)=\dfrac{L_j}{2}\lambda.
	\end{eqnarray*}
In an analogous way, we obtain
$$\bigg|\dfrac{r_j(x+\lambda e_i)-r_j(x-\lambda e_i)}{2\lambda}-\dfrac{\partial r_j(x)}{\partial x_i}\bigg|\leq \dfrac{M_j}{\beta_j+1}\lambda^{\beta_j},$$
this implies that 
$$	\|\nabla s_j(x)-g_{s_j}(x,\lambda)\|\leq\dfrac{\sqrt{n}L_j}{2}\lambda \quad \text{and} \quad	\|\nabla r_j(x)-g_{r_j}(x,\lambda)\|\leq\dfrac{\sqrt{n}M_j}{\beta_j+1}\lambda^{\beta_j}.$$
Therefore, using the triangle inequality, we obtain
	$$\|\nabla f_j(x)-g_j(x,\lambda)\|\leq\|\nabla s_j(x)-g_{s_j}(x,\lambda)\|+ \|\nabla r_j(x)-g_{r_j}(x,\lambda)\|\leq \dfrac{\sqrt{n}L_j}{2}\lambda+\dfrac{\sqrt{n}M_j}{\beta_j+1}\lambda^{\beta_j},$$	
	concluding the proof of the lemma.
\end{proof}

\begin{remark}\label{remk1}
If there exists $\mathcal{L}\subset  \mathcal{J}$ such that $\nabla f_j$ is Lipschitz for all $j\in \mathcal{L}$ and $\nabla f_j$ is H\"{o}lder for all $j \in  \mathcal{J} \setminus \mathcal{L}$, then $f_j$ satisfies the conditions of Lemma \ref{condLH}. In this case, it suffices to define
\begin{equation*}\label{def:sjrj}
s_j \coloneqq \left\{\hspace{-5pt}
\renewcommand{\arraystretch}{1.3}
\begin{tabular}{ll}
 $f_j$, &\hspace{-12pt} if $j \in \mathcal{L}$,\\
 $0$,   &\hspace{-12pt} if $j \in \mathcal{J} \setminus \mathcal{L}$,
\end{tabular}\right.
\quad \text{and} \quad
r_j \coloneqq \left\{\hspace{-5pt}
\renewcommand{\arraystretch}{1.3}
\begin{tabular}{ll}
 $f_j$, &\hspace{-12pt} if $j \in \mathcal{J} \setminus \mathcal{L}$,\\
 $0$,   &\hspace{-12pt} if $j \in \mathcal{L}$.
\end{tabular}\right.
\end{equation*}
\end{remark}
A scenario that satisfies Remark \ref{remk1} is given below. For each $j\in \mathcal{J}$, let us consider $f_j:\R^n\to\R$ and a $\mathcal{L}\subset \mathcal{J}$ where
\begin{equation}\label{LH}
f_j(x)\!=\!\left\{\hspace{-5pt}
\renewcommand{\arraystretch}{1.8}
\begin{tabular}{ll}
 $\dfrac{1}{2}\|A_j x - b_j\|_2^2$, &\hspace{-12pt} if $j \in \mathcal{L}$,\\
 $\dfrac{\mu_j}{p}\|D_j x\|_p^p$, & \hspace{-12pt} if $j \in \mathcal{J} \setminus \mathcal{L}$,
\end{tabular}\right.
\end{equation}
where $A_j \in \mathbb{R}^{m_j \times n}$, $b_j \in \mathbb{R}^{m_j}$, $D_j \in \mathbb{R}^{k_j \times n}$ is a linear operator, $\mu_j > 0$, and $1 < p < 2$. It is straightforward to verify that the quadratic functions $f_j(x) = \tfrac{1}{2}\|A_j x - b_j\|_2^2$ possess Lipschitz continuous gradients. Moreover, the functions $
f_j(x)=\frac{\mu_j}{p}\|D_jx\|_p^p
$ have $(p-1)$-Hölder continuous gradients. A proof of this fact can be found in \cite{Bernigaud2024}. The following Example \ref{ex1} presents a particular instance of the function defined in \eqref{LH}.

\begin{example}\label{ex1} 
Consider the function $f:\mathbb{R} \to \mathbb{R}^2$ defined as in \eqref{LH}, with $j=1,2$, where we take $A_1 = I$, $b_1 = 0$, $D_2 = I$, $\mu_2 = 1$, and $p=\tfrac{3}{2}$.
In this setting, $f_1(x) = \tfrac{1}{2}|x|^2$ and $f_2(x) = \tfrac{2}{3}|x|^{\tfrac{3}{2}}$. We claim that $\nabla f_2$ is H\"{o}lder continuous but not Lipschitz. Indeed, for $x \in \mathbb{R}$, 
\[
\nabla f_2(x) = \operatorname{sgn}(x) \cdot |x|^{\tfrac{1}{2}}.
\]
Taking $x = 0$ and $y = \tfrac{1}{n}$, where $n \in \mathbb{N}$, we obtain
\[
|\nabla f_2(x) - \nabla f_2(y)| = \left|\nabla f_2(0) - \nabla f_2\left(\frac{1}{n}\right)\right| = \left|0 - \frac{1}{\sqrt{n}}\right| = \frac{1}{\sqrt{n}}.
\]
If $\nabla f_2$ were Lipschitz, there would exist a constant $M > 0$ such that 
\[
|\nabla f_2(x) - \nabla f_2(y)| \leq M |x - y|,
\]
so we would have
\[
\frac{1}{\sqrt{n}} \leq M \left|0 - \frac{1}{n}\right| = \frac{M}{n},
\quad \text{i.e.,} \quad
\sqrt{n} \leq M.
\]
This would imply that the set of natural numbers is bounded above, which is absurd. Therefore, we conclude that the gradient $\nabla f_2$ is not Lipschitz.
\end{example}

The Assumptions \ref{a2} and \ref{sup.003} provide greater modeling flexibility than traditional approaches based solely on Lipschitz or H\"{o}lder continuity. Specifically, they allow each function $f_j$ to be decomposed into a sum of two components: one with a gradient satisfying the Lipschitz condition and the other with a gradient satisfying the H\"{o}lder condition. This decomposition encompasses a broader class of problems in which the smoothness of each $f_j$ is non-uniform. As a result, the method remains valid and effective regardless of whether each $f_j$ has a purely Lipschitz gradient, a purely H\"{o}lder gradient, or any structured combination of both. It is important to highlight that for each $f_j$ that has a gradient satisfying only Lipschitz conditions or only H\"{o}lder conditions, Assumption \ref{a2} is also satisfied.

The next lemmas are useful in demonstrating that the sufficient descent condition at Step~4 is satisfied for sufficiently large $\sigma_k$, as stated in Theorem~\ref{bdef}.

\begin{lemma}\label{BD}
	If $\bar{x}^k$ is a solution of problem \eqref{probminmax}, then there exist $\gamma\in~\R_+^m,\,\dsum_{j=1}^{m}\gamma_{j}=1$ and $w_j^k\in\partial h_j(\bar{x}^{k}),$ such that	\begin{equation}\label{sup.001}
		\langle g_{f_j}(x^k,\lambda_k),\bar{x}^k-x^k\rangle+\dfrac{1}{2}\langle B_j^k(\bar{x}^k-x^k),\bar{x}^k-x^k\rangle+\dfrac{\sigma_k}{2}\|\bar{x}^k-x^k\|^2+h_j(\bar{x}^{k})\leq h_ j(x^k),
	\end{equation}
	for all $j\in \mathcal{J}$, and 
	\begin{equation}\label{cond1}
		\|\sum_{j=1}^m\gamma_j
		\big[g_{f_j}(x^k,\lambda_k)+w_j^{k}+B_j^k(\Bar{x}^k-x^k)\big]+\sigma_k (\bar{x}^k-x^k)\|=0.
	\end{equation}
\end{lemma}

\begin{proof}
	Let $\bar{x}^k$ be the solution of problem \eqref{probminmax}, then $\Psi_{\sigma_k}^{x^k}(\bar{x}^k)\leq \Psi_{\sigma_k}^{x^k}(x)$ for all $x\in\R^n$, whence it follows that $\Psi_{\sigma_k}^{x^k}(\bar{x}^k)\leq \Psi_{\sigma_k}^{x^k}(x^k)=0$, that is,
	$$\max_{j\in \mathcal{J}} \big[\langle g_{f_j}(x^k,\lambda_k),\bar{x}^k-x^k\rangle+\tfrac{1}{2}\langle B^k_j(\bar{x}^k-x^k),\bar{x}^k-x^k\rangle+h_j(\bar{x}^k)-h_j(x^k)\big]+\tfrac{{\sigma_k}}{2}\|\bar{x}^k-x^k\|^2\leq 0,$$ 
	which implies
	$$\langle g_{f_j}(x^k,\lambda_k),\bar{x}^k-x^k\rangle+\tfrac{1}{2}\langle B^k_j(\bar{x}^k-x^k),\bar{x}^k-x^k\rangle+h_j(\bar{x}^k)-h_j(x^k)+\tfrac{{\sigma_k}}{2}\|\bar{x}^k-x^k\|^2\leq 0,\quad j\in \mathcal{J}.$$
	This proves condition \eqref{sup.001}.
	
	To prove condition \eqref{cond1}, we follow the same steps taken by Calderón \cite{teselizet} in the proof of Theorem~2.1. We can see that solving problem \eqref{probminmax} is equivalent to solving the following problem:
	\begin{eqnarray}\label{probzp}
		&&\min_{x\in\R^n,z\in\R} \,z+\tfrac{\sigma_k}{2}\|x-x^k\|^2\\
		&&\text{s.t.}\;\;\langle g_{f_j}(x^k, \lambda_k),x-x^k\rangle+\tfrac{1}{2}\langle B_j^k(x-x^k),(x-x^k)\rangle+h_j(x)-h_j(x^k)\leq z,\quad \forall\,j\in \mathcal{J}.\nonumber
	\end{eqnarray}
	Since $x^k\in\mathcal{C}$ and $h_j=\widetilde{h}_j$ on $\mathcal{C}$, problem \eqref{probzp} is a convex program over $\mathcal{C}\times\mathbb{R}$ with finite constraint functions, for which $(x^k,1)$ is a Slater point.
	If $\bar{x}^k$ is a solution of \eqref{probzp}, then by the Karush–Kuhn–Tucker (KKT) conditions \cite[Corollary~28.2.1 and Theorems~23.8 and 28.3]{Rockafellar1970} and Proposition~\ref{prop:stationary}, there exist $\gamma\in\R_+^m$ and $w_j^k\in\partial h_j(\bar{x}^k)$ such that 
	\[
\begin{aligned}
&\begin{pmatrix}
1 \\
\sigma_k(\bar{x}^k-x^k)
\end{pmatrix}
+\sum_{j=1}^m\gamma_j
\begin{pmatrix}
-1 \\
g_{f_j}(x^k, \lambda_k)+B_j^k(\bar{x}^k-x^k)+w_j^k
\end{pmatrix}
=
\begin{pmatrix}
0 \\
0
\end{pmatrix},\\[6pt]
&\gamma_j\Big(
\langle g_{f_j}(x^k, \lambda_k), \bar{x}^k-x^k\rangle
+\tfrac{1}{2}\langle B_j^k(\bar{x}^k-x^k),\bar{x}^k-x^k\rangle
+h_j(\bar{x}^k)-h_j(x^k)-z
\Big)=0,\quad j\in \mathcal{J}.
\end{aligned}
\]
From which it follows that
\begin{equation}\label{qq}
	\sigma_k (\bar{x}^k-x^k)=-\sum_{j=1}^m\gamma_j
	\big(g_{f_j}(x^k, \lambda_k)+w_j^k+B_j^k(\bar{x}^k-x^k)\big),\quad \sum_{j=1}^m\gamma_j=1.
\end{equation}
The condition \eqref{cond1} follows directly from \eqref{qq}. This concludes the proof of Lemma~\ref{BD}.
\end{proof}
Throughout this work, we denote 
\begin{equation}\label{thetak}
 \theta_k \coloneqq 1+\max_{j\in\mathcal{J}}\operatorname{tr}(B_j^k).
\end{equation}
\begin{lemma}\label{condsuf01}
	Let $\bar{x}^k$ be the point computed in Step~2. Suppose that Assumptions \ref{a2} and \ref{sup.003} hold. If 
	\begin{equation*}\label{estsk}
		\sigma_k\|\bar{x}^k-x^k\|\geq \frac{\epsilon}{3\theta_k},
	\end{equation*}
	then
	\begin{equation*}\label{dec02}
		F_j(\bar{x}^{k}) \leq F_j(x^k)
   -\bigg(\dfrac{\sigma_k-L_j(1+3\theta_k)}{2}\bigg)\|s^k\|^{2}
      +\dfrac{n^{\tfrac{1-\beta_j}{2}} M_j3^{\beta_j}\theta_k^{\beta_j}+M_j}{\beta_j+1}\|s^k\|^{\beta_j+1},\quad j\in \mathcal{J},
	\end{equation*}
	where $s^k\coloneqq\bar{x}^{k}-x^k$.
\end{lemma}

\begin{proof} 
By Assumption \ref{a2}, we have 
\begin{equation}\label{lip1}
	f_j(\bar{x}^{k})\leq f_j(x^k)+\langle \nabla f_j(x^k),s^k\rangle+ \frac{L_j}{2} \|s^k\|^{2}+\dfrac{M_j}{\beta_j+1} \|s^k\|^{\beta_j+1},\quad j\in \mathcal{J}.	
\end{equation}
Using \eqref{sup.001}, Assumption~\ref{sup.003}, and \eqref{lip1}, and $\langle B_j^ks^k,s^k\rangle\geq 0$, since $B_j^k\succeq 0$, we obtain
\begin{align*}
	f_j(\bar{x}^{k})+h_j(\bar{x}^{k})
    &\leq f_j(x^k)+h_j(\bar{x}^{k})
     +\langle g_{f_j}(x^k,\lambda_k),s^k\rangle
     +\tfrac{1}{2}\langle B_j^ks^k,s^k\rangle
     +\tfrac{\sigma_k}{2}\|s^k\|^2
     +\tfrac{L_j}{2} \|s^k\|^{2} \\
	&\quad + \langle \nabla f_j(x^k)-g_{f_j}(x^k,\lambda_k),s^k\rangle
     +\dfrac{M_j}{\beta_j+1} \|s^k\|^{\beta_j+1}
     -\tfrac{1}{2}\langle B_j^ks^k,s^k\rangle
     -\tfrac{\sigma_k}{2}\|s^k\|^2 \\
	&\leq f_j(x^k)+h_j(x^k)
     +\|\nabla f_j(x^k)-g_{f_j}(x^k,\lambda_k)\|\|s^k\|\\
     &+\tfrac{L_j}{2} \|s^k\|^{2}
     -\tfrac{\sigma_k}{2}\|s^k\|^2
     + \dfrac{M_j}{\beta_j+1} \|s^k\|^{\beta_j+1}\\
	&\leq F_j(x^k)
     +\sqrt{n}\dfrac{L_j}{2}\lambda_k\|s^k\|
     +\sqrt{n}\dfrac{M_j}{\beta_j+1}\lambda_k^{\beta_j}\|s^k\|
     + \dfrac{M_j}{\beta_j+1} \|s^k\|^{\beta_j+1}\\
    &\quad -\frac{\sigma_k-L_j}{2}\|s^k\|^2.
\end{align*}
By $0\leq\lambda_k\leq \dfrac{\epsilon}{\sigma_k\sqrt{n}}$ and the last inequality we have that
\begin{align*}
F_j(\bar{x}^{k})
&\leq F_j(x^k)
   +\sqrt{n}\,\dfrac{L_j}{2}\dfrac{\epsilon}{\sigma_k\sqrt{n}}\|s^k\|
   +\sqrt{n}\,\dfrac{M_j}{\beta_j+1}\dfrac{\epsilon^{\beta_j}}{\sigma_k^{\beta_j}\,(\sqrt{n})^{\beta_j}}\|s^k\| \\
&\quad -\bigg(\dfrac{\sigma_k-L_j}{2}\bigg)\|s^k\|^{2}
   +\dfrac{M_j}{\beta_j+1}\|s^k\|^{\beta_j+1} \\[6pt]
&\leq F_j(x^k)
   +\dfrac{L_j}{2}\dfrac{\epsilon}{\sigma_k}\|s^k\|
   +n^{\tfrac{1-\beta_j}{2}}\dfrac{M_j}{\beta_j+1}\dfrac{\epsilon^{\beta_j}}{\sigma_k^{\beta_j}}\|s^k\| \\
&\quad -\bigg(\dfrac{\sigma_k-L_j}{2}\bigg)\|s^k\|^{2}
   +\dfrac{M_j}{\beta_j+1}\|s^k\|^{\beta_j+1}.
\end{align*}
  Using $\sigma_k\|s^k\|\geq\frac{\epsilon}{3\theta_k}$ in the last inequality we obtain
\begin{align*}
F_j(\bar{x}^{k})
&\leq F_j(x^k)
   +\dfrac{L_j}{2}{3\theta_k}\|s^k\|\|s^k\|
   +n^{\tfrac{1-\beta_j}{2}}{3^{\beta_j}\theta_k^{\beta_j}}\dfrac{M_j}{\beta_j+1}\|s^k\|^{\beta_j}\|s^k\| \\
&\quad -\bigg(\dfrac{\sigma_k-L_j}{2}\bigg)\|s^k\|^{2}
   +\dfrac{M_j}{\beta_j+1}\|s^k\|^{\beta_j+1} \\[6pt]
&\leq F_j(x^k)
   +\dfrac{L_j}{2}{3\theta_k}\|s^k\|^2
   +n^{\tfrac{1-\beta_j}{2}}{3^{\beta_j}\theta_k^{\beta_j}}\dfrac{M_j}{\beta_j+1}\|s^k\|^{\beta_j+1} \\
&\quad -\bigg(\dfrac{\sigma_k-L_j}{2}\bigg)\|s^k\|^{2}
   +\dfrac{M_j}{\beta_j+1}\|s^k\|^{\beta_j+1} \\[6pt]
&\leq F_j(x^k)
   -\bigg(\dfrac{\sigma_k-L_j(1+3\theta_k)}{2}\bigg)\|s^k\|^{2}
      +\dfrac{n^{\tfrac{1-\beta_j}{2}} M_j3^{\beta_j}\theta_k^{\beta_j}+M_j}{\beta_j+1}\|s^k\|^{\beta_j+1}.
\end{align*}
This concludes the proof of Lemma~\ref{condsuf01}.
\end{proof}

In Theorem~\ref{bdef} below we prove that there exists a $\sigma_k$ such that the decrement given in \eqref{eq.te} is achieved, thus ensuring the well-definedness of Step~4.
		
\begin{theorem}\label{bdef}
	Let $\bar{x}^k$ be the point computed in Step~2 of Algorithm~\ref{algfree}. Suppose that Assumptions \ref{a2} and \ref{sup.003} hold. If 
	\begin{equation}\label{estgrad12}
		\sigma_k\|\bar{x}^k-x^k\|\geq \frac{\epsilon}{3\theta_k},
	\end{equation}
	and  
	\begin{equation}\label{sig}
		\sigma_k\geq \max_{j\in \mathcal{J}}\bigg[\dfrac{L_j(1+3\theta_k)}{1-\alpha}
+\dfrac{2n^{\tfrac{1-\beta_j}{2}} M_j3^{\beta_j}\theta_k^{\beta_j}+2M_j}{(\beta_j+1)(1-\alpha)}{(3\theta_k)^{1-\beta_j}}
\epsilon^{\beta_j-1}\bigg]^{\tfrac{1}{\beta_j}},
	\end{equation}
	then
	\begin{equation}\label{decf}
		F_j(\bar{x}^{k})\leq F_j(x^k)-\dfrac{\alpha}{2}\sigma_k\|\bar{x}^{k}-x^k\|^2,\quad j\in \mathcal{J}.
	\end{equation}
\end{theorem}

\begin{proof} 
We define $s^k \coloneqq \bar{x}^{k}-x^k$. 
In Lemma~\ref{condsuf01}, we obtain
\begin{equation*}\label{d11}
	F_j(\bar{x}^{k}) \leq F_j(x^k)
   -\bigg(\dfrac{\sigma_k-L_j(1+3\theta_k)}{2}\bigg)\|s^k\|^{2}
      +\dfrac{n^{\tfrac{1-\beta_j}{2}} M_j3^{\beta_j}\theta_k^{\beta_j}+M_j}{\beta_j+1}\|s^k\|^{\beta_j+1}.
\end{equation*}
Thus, taking this inequality into consideration, we can conclude that to prove inequality \eqref{decf} it is necessary to show that
\[
-\dfrac{\alpha}{2}\sigma_k\|s^k\|^2 \;\geq\;
-\bigg(\dfrac{\sigma_k-L_j(1+3\theta_k)}{2}\bigg)\|s^k\|^{2}
      +\dfrac{n^{\tfrac{1-\beta_j}{2}} M_j3^{\beta_j}\theta_k^{\beta_j}+M_j}{\beta_j+1}\|s^k\|^{\beta_j+1},
\]
which is equivalent to proving the inequality
\[
\sigma_k \;\geq\; \dfrac{L_j(1+3\theta_k)}{1-\alpha}
+\dfrac{2n^{\tfrac{1-\beta_j}{2}} M_j3^{\beta_j}\theta_k^{\beta_j}+2M_j}{(\beta_j+1)(1-\alpha)}\|s^k\|^{\beta_j-1}.
\]
By \eqref{estgrad12}, we have that $\dfrac{\epsilon}{3\theta_k\sigma_k}\leq \|s^k\|$, which implies
\begin{align*}
&\dfrac{L_j(1+3\theta_k)}{1-\alpha}
+\dfrac{2n^{\tfrac{1-\beta_j}{2}} M_j3^{\beta_j}\theta_k^{\beta_j}+2M_j}{(\beta_j+1)(1-\alpha)}\|s^k\|^{\beta_j-1} \\[6pt]
&\quad \leq \dfrac{L_j(1+3\theta_k)}{1-\alpha}
+\dfrac{2n^{\tfrac{1-\beta_j}{2}} M_j3^{\beta_j}\theta_k^{\beta_j}+2M_j}{(\beta_j+1)(1-\alpha)}
\bigg(\dfrac{\epsilon}{3\theta_k\sigma_k}\bigg)^{\beta_j-1} \\[6pt]
&\quad \leq \dfrac{L_j(1+3\theta_k)}{1-\alpha}
+\dfrac{2n^{\tfrac{1-\beta_j}{2}} M_j3^{\beta_j}\theta_k^{\beta_j}+2M_j}{(\beta_j+1)(1-\alpha)}{(3\theta_k)^{1-\beta_j}}
\epsilon^{\beta_j-1}\sigma_k^{1-\beta_j} \\[6pt]
&\quad \leq \left[
\dfrac{L_j(1+3\theta_k)}{1-\alpha}
+\dfrac{2n^{\tfrac{1-\beta_j}{2}} M_j3^{\beta_j}\theta_k^{\beta_j}+2M_j}{(\beta_j+1)(1-\alpha)}{(3\theta_k)^{1-\beta_j}}
\epsilon^{\beta_j-1}\right]\sigma_k^{1-\beta_j} \\[6pt]
&\quad \leq \sigma_k^{\beta_j}\,\sigma_k^{1-\beta_j}
= \sigma_k.
\end{align*}
This completes the proof of \eqref{decf}.
\end{proof}

\begin{remark} Note that, for each fixed iteration $k$, all quantities appearing in the lower bound on $\sigma_k$ given in Theorem \ref{bdef} remain unchanged during the inner loop of Algorithm \ref{algfree}. Therefore, by successively doubling $\sigma_k$ in Step 4, this bound will eventually be satisfied. Consequently, the sufficient decrease condition \eqref{eq.te} is achieved after finitely many updates of $\sigma_k$, ensuring the well-definedness of Step 4.
\end{remark}

Theorem~\ref{bdef} guarantees that, as long as the test point $\bar{x}^{k}$ does not satisfy the stopping criterion described in Step~$3$, there will be a decrease in the value of the objective function, with a reduction established in Step~$4$. This reduction is fundamental for obtaining the convergence and complexity results that will be presented in the next section.

The following proposition bounds the Pareto criticality residual at the point returned by Step~$3$ and gives conditions, including an upper bound on $\lambda_k$, under which this point is an $\epsilon$-Pareto critical point in the sense of Definition~\ref{epspareto}.

\begin{proposition}\label{prop:epspareto}
    Suppose that Assumptions~\ref{a2} and~\ref{sup.003}
hold, and let $\bar{x}^k$ be a solution of
subproblem~\eqref{probminmax}.
Let $\theta_k$ be given by~\eqref{thetak},
and assume that $B_j^k\succeq 0$ for all
$j\in\mathcal J$ and $\sigma_k\geq 1$.
Then the following statements hold.
\begin{enumerate}
\item[(a)]
If Algorithm~\ref{algfree} terminates at Step~3,
then there exist $\gamma_j\geq 0$, with
$\sum_{j=1}^m\gamma_j=1$, such that
\begin{align}
d\bigg(
0,\sum_{j=1}^m\gamma_j
\big(\nabla f_j(\bar{x}^k)
+\partial h_j(\bar{x}^k)\big)
\bigg)
&<
\sum_{j=1}^m\gamma_j
\left[
L_j\frac{\epsilon}{3\sigma_k\theta_k}
+M_j\left(
\frac{\epsilon}{3\sigma_k\theta_k}
\right)^{\beta_j}
\right]\nonumber\\
&\quad+
\sum_{j=1}^m\gamma_j
\left[
\frac{\sqrt n L_j}{2}\lambda_k
+\frac{\sqrt n M_j}{1+\beta_j}
\lambda_k^{\beta_j}
\right]
+\frac{\epsilon}{3}.
\label{eq:pareto-residual}
\end{align}

\item[(b)]
Let $0<\epsilon\leq 1$ and define
$
\bar L=\max_{j\in\mathcal J}L_j,$ $
\bar M=\max_{j\in\mathcal J}M_j,$ $
\beta=\min_{j\in\mathcal J}\beta_j.
$
Suppose, in addition, that
$$
0\leq\lambda_k\leq
\min\left\{
1,\,
\frac{\epsilon}{8\sqrt n\,\bar L},\,
\left(
\frac{\epsilon}{8\sqrt n\,\bar M}
\right)^{1/\beta}
\right\},
$$
and that, for every $j\in\mathcal J$,
$$
L_j\leq\frac{\sigma_k\theta_k}{2}
\quad\text{and}\quad
M_j\leq
\frac{\epsilon}{6}
\left(
\frac{\epsilon}{3\sigma_k\theta_k}
\right)^{-\beta}.
$$
Then the point $\bar{x}^k$ returned by
Algorithm~\ref{algfree} is an
$\epsilon$-Pareto critical point
in the sense of Definition~\ref{epspareto}.
\end{enumerate}
\end{proposition}
\begin{proof}
By Lemma~\ref{BD}, there exist
$\gamma_j\geq 0$, with
$\sum_{j=1}^m\gamma_j=1$, and
$w_j^k\in\partial h_j(\bar{x}^k)$ such that
$$
\sum_{j=1}^m\gamma_j
\left[
g_{f_j}(x^k,\lambda_k)+w_j^k
+B_j^k(\bar{x}^k-x^k)
\right]
+\sigma_k(\bar{x}^k-x^k)=0.
$$
Setting $s^k=\bar{x}^k-x^k$, we obtain
\begin{align*}
\bigg\|
\sum_{j=1}^m\gamma_j
\big(\nabla f_j(\bar{x}^k)+w_j^k\big)
\bigg\|&\leq
\sum_{j=1}^m\gamma_j
\|\nabla f_j(\bar{x}^k)-\nabla f_j(x^k)\|\\
&+
\sum_{j=1}^m\gamma_j
\|\nabla f_j(x^k)-g_{f_j}(x^k,\lambda_k)\|\\
&
+\left(
\max_{j\in\mathcal J}\|B_j^k\|+\sigma_k
\right)\|s^k\|.
\end{align*}
Using Assumptions~\ref{a2} and~\ref{sup.003},
we deduce that
\begin{align*}
d\bigg(
0,\sum_{j=1}^m\gamma_j
\big(\nabla f_j(\bar{x}^k)
+\partial h_j(\bar{x}^k)\big)
\bigg)
&\leq
\sum_{j=1}^m\gamma_j
\big(
L_j\|s^k\|+M_j\|s^k\|^{\beta_j}
\big)\\
&+
\sum_{j=1}^m\gamma_j
\left(
\frac{\sqrt n L_j}{2}\lambda_k
+\frac{\sqrt n M_j}{1+\beta_j}
\lambda_k^{\beta_j}
\right)\\
&+
\left(
\max_{j\in\mathcal J}\|B_j^k\|
+\sigma_k
\right)\|s^k\|.
\end{align*}   
\noindent\textbf{(a)}
Since Algorithm~\ref{algfree} terminates
at Step~3, we have $
\sigma_k\|s^k\|
<
\frac{\epsilon}{3\theta_k}.
$
Moreover, since $B_j^k\succeq 0$,
$\|B_j^k\|\leq\operatorname{tr}(B_j^k)$.
Thus, using $\sigma_k\geq 1$, we obtain
\begin{align*}
\bigg(
\max_{j\in\mathcal J}\|B_j^k\|
+\sigma_k
\bigg)\|s^k\| &\leq
\left(
\theta_k-1+\sigma_k
\right)\|s^k\|\\
&<
\frac{\theta_k-1+\sigma_k}{\sigma_k\theta_k}
\frac{\epsilon}{3}
\leq\frac{\epsilon}{3}.
\end{align*}
Furthermore,
$$
\|s^k\|<
\frac{\epsilon}{3\sigma_k\theta_k}.
$$
Substituting these estimates into the
preceding inequality yields
\eqref{eq:pareto-residual}.

\medskip
\noindent\textbf{(b)}
Under the stated conditions, we have
$$
\sum_{j=1}^m\gamma_j
L_j\frac{\epsilon}{3\sigma_k\theta_k}
\leq\frac{\epsilon}{6}.
$$
Since $0<\epsilon\leq 1$,
$\sigma_k\theta_k\geq 1$, and
$\beta_j\geq\beta$, it follows that
$$
\left(
\frac{\epsilon}{3\sigma_k\theta_k}
\right)^{\beta_j}
\leq
\left(
\frac{\epsilon}{3\sigma_k\theta_k}
\right)^\beta.
$$
Consequently,
$$
\sum_{j=1}^m\gamma_jM_j
\left(
\frac{\epsilon}{3\sigma_k\theta_k}
\right)^{\beta_j}
\leq\frac{\epsilon}{6}.
$$
On the other hand, the choice of
$\lambda_k$ implies
$$
\sum_{j=1}^m\gamma_j
\frac{\sqrt nL_j}{2}\lambda_k
\leq\frac{\epsilon}{16}.
$$
Since $\lambda_k\leq 1$ and
$\beta_j\geq\beta$, we also have
$$
\sum_{j=1}^m\gamma_j
\frac{\sqrt nM_j}{1+\beta_j}
\lambda_k^{\beta_j}
\leq\frac{\epsilon}{8}.
$$
Combining these estimates with
\eqref{eq:pareto-residual}, we obtain
\begin{align*}
d\bigg(
0,\sum_{j=1}^m\gamma_j
\big(\nabla f_j(\bar{x}^k)
+\partial h_j(\bar{x}^k)\big)
\bigg)<
\frac{\epsilon}{6}
+\frac{\epsilon}{6}
+\frac{\epsilon}{16}
+\frac{\epsilon}{8}
+\frac{\epsilon}{3}
<\epsilon.
\end{align*}
Therefore, $\bar{x}^k$ is an
$\epsilon$-Pareto critical point.
\end{proof}

\begin{remark}\label{cparada1}
Proposition~\ref{prop:epspareto} relates the stopping criterion adopted in
Step~3 of Algorithm~\ref{algfree} to approximate Pareto criticality. In particular,
item~(a) establishes an upper bound on the Pareto
criticality residual at the returned point, while
item~(b) provides sufficient conditions under which
this residual is smaller than the prescribed
tolerance $\epsilon$. These conditions involve the constants $L_j$, $M_j$ and
$\beta_j$ of Assumption~\ref{a2}, which Algorithm~\ref{algfree}
does not use. Hence, the proposition does not ensure that every
point returned at Step~3 is $\epsilon$-Pareto critical; it
identifies conditions under which this holds.
The condition on $\lambda_k$ in item~(b) is
compatible with Step~1 of Algorithm~\ref{algfree}: both are
upper bounds on $\lambda_k$, so it suffices to choose
$\lambda_k$ not larger than the smaller of the two.
\end{remark}

If $\bar{x}^k$ solves subproblem~\eqref{probminmax} only up to a residual $r_k$ in the optimality condition~\eqref{cond1}, then $\|r_k\|$ is added to the right-hand side of~\eqref{eq:pareto-residual}.

The Section~\ref{seccomplex} is dedicated to establishing upper bounds on the number of iterations and the number of function evaluations required to reach an approximate solution within a predefined tolerance.

\section{ Complexity Analysis}\label{seccomplex}
 In this section, we establish  worst-case complexity bounds in terms of the number of function evaluations and iterations for Algorithm~\ref{algfree}. We show, in Theorem~\ref{bdef}, that the sufficient decrease condition required in Step~4 is satisfied for all sufficiently large values of $\sigma_k$. In the same theorem, we derive a lower bound on the decrease in each objective function $F_j$, $j\in\mathcal{J}$, in terms of $\sigma_k$, $\|x^{k+1}-x^k\|$, and the algorithmic parameters. These results provide the foundation for the subsequent  worst-case iteration complexity analysis.

 We observe that, for the well-definedness of the Algorithm~\ref{algfree}, the only requirement imposed on the matrices $B_j^k$ is that they remain positive semidefinite. However, in order to establish complexity results, it is necessary to prove that the method terminates after a finite number of iterations. Consequently, the sequence of parameters $\{\sigma_k\}$ must be bounded, which requires the sequence $\{\operatorname{tr}(B_j^k)\}$ to be bounded. Thus, we make the following assumption.
  \begin{A}\label{A2}
	There exists $\overline{B}>
    0$ such that $\operatorname{tr}(B_j^k)\leq\overline{B}$ for all iterations $k$ of Algorithm~\ref{algfree} and $j\in \mathcal{J}$.
\end{A}
To simplify some expressions that appear subsequently, we define
\begin{equation}
    \overline{\sigma}_j(b)=\bigg[\dfrac{L_j\big(1+3(1+b)\big)}{1-\alpha}
+\dfrac{2n^{\tfrac{1-\beta_j}{2}} M_j\big(3(1+b)\big)^{\beta_j}+2M_j}{(\beta_j+1)(1-\alpha)}\big(3(1+b)\big)^{1-\beta_j}
\bigg]^{\tfrac{1}{\beta_j}},
\end{equation}
for $b\geq 0$.

In Theorem \ref{tedec1} below, we prove that each $F_j$ is reduced by a constant factor that depends only on the constants of the method and on Assumptions~\ref{a2} and~\ref{A2}.

\begin{theorem}\label{tedec1}
Suppose that Assumptions \ref{sup.003}, \ref{a2}, and \ref{A2} are satisfied, and let $x^{k+1}$ be the output of Algorithm~\ref{algfree}. Then the minimum reduction in the objective function is given by
\begin{equation}\label{eq0001.te} 
    F_j(x^{k+1})
\leq F_j(x^k)-\dfrac{\alpha}{18c(1+\overline{B})^2}\,\epsilon^{\tfrac{\beta+1}{\beta}},\quad j\in \mathcal{J}.,
\end{equation}
where 
\[
\beta = \min\{\beta_j : j \in \mathcal{J}\}
\quad \text{and} \quad
c = 2\, \max \bigg\{\sigma_0,\,\max_{j\in \mathcal{J}}\overline{\sigma}_j(\overline{B})\bigg\}.
\]
\end{theorem}

\begin{proof}
In Theorem \ref{bdef} we proved that, if
	
$\sigma_k$ satisfies \eqref{sig}
and the method does not stop at Step~3, then

	\begin{equation}\label{decf1}
		F_j(x^{k+1})\leq F_j(x^k)-\dfrac{\alpha}{2}\sigma_k\|x^{k+1}-x^k\|^2,\quad j\in \mathcal{J}.
	\end{equation}
    Since the method does not stop at Step~3, and by Assumption \ref{A2}, we have
$$
\sigma_k\|x^{k+1}-x^k\|\geq \frac{\epsilon}{3\big(1+\max_{j\in\mathcal{J}}\operatorname{tr}(B_j^k)\big)}\geq \frac{\epsilon}{3(1+\overline{B})}.
$$
Substituting this inequality into \eqref{decf1}, we obtain
\begin{equation}\label{decf11}
		F_j(x^{k+1})\leq F_j(x^k)-\dfrac{\alpha}{18\sigma_k(1+\overline{B})^2}\epsilon^2,\quad j\in \mathcal{J}.
	\end{equation}
Therefore, since $\theta_k\leq 1+\overline{B}$ by Assumption~\ref{A2}, $0<\epsilon<1$ and $\beta\leq\beta_j$, the right-hand side of \eqref{sig} is at most $\max_{j\in\mathcal{J}}\overline{\sigma}_j(\overline{B})\,\epsilon^{\tfrac{\beta-1}{\beta}}$, uniformly in $k$. Since $\sigma_k$ is doubled only when \eqref{eq.te} fails and is never decreased, we can conclude that
$$
\sigma_k
\leq  2\,\epsilon^{\tfrac{\beta-1}{\beta}} \max \bigg\{\sigma_0,\,\max_{j\in \mathcal{J}}\overline{\sigma}_j(\overline{B})\bigg\}=\sigma_{\max}.
$$
This implies that $\sigma_{\max}= c\epsilon^{\tfrac{\beta-1}{\beta}}$.  
By substituting $\sigma_{\max}$ in \eqref{decf11}, we conclude that 
\begin{align*}	
F_j(x^{k+1})
\leq F_j(x^k)-\dfrac{\alpha}{18c(1+\overline{B})^2}\,\epsilon^{\tfrac{\beta+1}{\beta}},\quad j\in \mathcal{J}.
\end{align*}
This completes the proof of \eqref{eq0001.te}.
\end{proof}

In the inequality \eqref{eq0001.te} of Theorem~\ref{tedec1}, we can observe that, as $\beta$ approaches zero, the expected reduction in each $F_j$ decreases. Therefore, it is expected that the smaller $\beta$ is, the greater will be the number of iterations required to reach an approximate solution with a given tolerance.

Next, we demonstrate one of the main results of this section, showing that the number of iterations required to reach an approximate solution with a given tolerance is bounded by a multiple of $\epsilon^{-\tfrac{\beta+1}{\beta}}$.  
Theorems~\ref{complexity} and~\ref{te.3}, presented below, are adaptations of Theorems~4.2 and~4.3 presented in \cite{amaral2025complexity}.

The following additional assumption is necessary in this context:

\begin{A}\label{A:Fbounded}
There exist $j_0\in\mathcal{J}$ and $F^{\mathrm{low}}_{j_0}\in\mathbb{R}$ such that $F_{j_0}(x)\geq F^{\mathrm{low}}_{j_0}$ for all $x\in\mathcal{C}$.
\end{A}
In the sequence, $\overline{F}\in\mathbb{R}^m$ is a predefined reference vector representing target levels for the values of the objective function $F$, it does not necessarily correspond to lower bounds or optimal values.

\begin{theorem}\label{complexity}
Suppose that Assumptions \ref{a2}, \ref{sup.003}, \ref{A2} and \ref{A:Fbounded} hold. Then, for any given $\overline{F}\in\mathbb{R}^m$, the number of iterations satisfying $\sigma_k\|x^{k+1}-x^k\|\geq \frac{\epsilon}{3\big(1+\max_{j\in\mathcal{J}}  \operatorname{tr}(B_j^k)\big)}$ and $F(x^k)> \overline{F}$ is bounded above by
\begin{equation*}\label{complex1}
	\displaystyle{1+\max\left\{0,\,F_{j_0}(x^{0})-\max\{\overline{F}_{j_0},\,F^{\mathrm{low}}_{j_0}\}\right\}\dfrac{18c(1+\overline{B})^2}{\alpha}\epsilon^{-\tfrac{\beta+1}{\beta}},}
\end{equation*}
where $c$ and $\beta$ are the same as those given in Theorem~\ref{tedec1}.
\end{theorem}

\begin{proof} 
The case in which $\overline{F}$ does not satisfy $F(x^0)> \overline{F}$ is trivial. It remains to analyze the case in which $F(x^0)> \overline{F}$. In Theorem~\ref{tedec1}, we proved that if Algorithm~\ref{algfree} does not meet the established stopping criterion at iteration $k$, then the condition
\begin{equation*}\label{eq01.te} 
F_j(x^{i}) \leq F_j(x^{i-1}) - \dfrac{\alpha}{18c(1+\overline{B})^2}\,\epsilon^{\tfrac{\beta+1}{\beta}}, 
\quad j\in \mathcal{J},
\end{equation*}
holds for every $i=1,2,\ldots,k$.  
Hence, if $F(x^k) > \overline{F}$, we have that
\[
\overline{F}_j < F_j(x^{k-1}) - \dfrac{\alpha}{18c(1+\overline{B})^2}\,\epsilon^{\tfrac{\beta+1}{\beta}},
\]
\[
\overline{F}_j < F_j(x^{k-2}) - 2\dfrac{\alpha}{18c(1+\overline{B})^2}\,\epsilon^{\tfrac{\beta+1}{\beta}},
\]
\[
\vdots
\]
\[
\overline{F}_j < F_j(x^{0}) - k\,\dfrac{\alpha}{18c(1+\overline{B})^2}\,\epsilon^{\tfrac{\beta+1}{\beta}}.
\]
In particular, for $j=j_0$, and since $F_{j_0}(x^k)\geq F^{\mathrm{low}}_{j_0}$ by Assumption~\ref{A:Fbounded},
$$\overline{F}_{j_0} < F_{j_0}(x^{0}) - k\,\dfrac{\alpha}{18c(1+\overline{B})^2}\,\epsilon^{\tfrac{\beta+1}{\beta}}\quad\text{and}\quad F^{\mathrm{low}}_{j_0} \leq F_{j_0}(x^{0}) - k\,\dfrac{\alpha}{18c(1+\overline{B})^2}\,\epsilon^{\tfrac{\beta+1}{\beta}}.$$
It follows that
$$\max\{\overline{F}_{j_0},\,F^{\mathrm{low}}_{j_0}\} \leq F_{j_0}(x^{0}) - k\,\dfrac{\alpha}{18c(1+\overline{B})^2}\,\epsilon^{\tfrac{\beta+1}{\beta}}.$$
This inequality implies that
\[
k \leq \dfrac{18c(1+\overline{B})^2\big(F_{j_0}(x^{0})-\max\{\overline{F}_{j_0},\,F^{\mathrm{low}}_{j_0}\}\big)}{\alpha}\,
\epsilon^{-\tfrac{\beta+1}{\beta}}{.}
\]
This concludes the proof of the theorem.
\end{proof}

It is important to note that Theorem~\ref{complexity} does not account for the function evaluations performed in Step 4, nor for those used to construct the gradient approximations in Step 1. This issue is related to the number of updates of the regularization parameter $\sigma_k$ and will be properly addressed in Theorem~\ref{te.3}.

\begin{theorem}\label{te.3}
	Suppose that the assumptions of Theorem~\ref{complexity} hold and that the gradient approximations in Step~1 are computed using finite differences. Then Algorithm~\ref{algfree} stops after finitely many iterations and the total number of evaluations of the smooth components $f_j$
used in finite-difference approximations and of the objective
components $F_j$ used in Step~4 is at most
	\begin{equation*}\label{03.1} N_f\leq m(2n+1)\left[
1+\left(F_{j_0}(x^0)-F^{\mathrm{low}}_{j_0}\right)
\frac{18c(1+\overline B)^2}{\alpha}
\epsilon^{-\frac{\beta+1}{\beta}}+
\log_2\left(\frac{\sigma_{\max}}{\sigma_0}\right)+1
\right],
	\end{equation*}
	where
	
\[\sigma_{\max}=2\,\epsilon^{\tfrac{\beta-1}{\beta}} \max \bigg\{\sigma_0,\,\max_{j\in \mathcal{J}}\overline{\sigma}_j(\overline{B})\bigg\}\] with $c$ and $\beta$ the same as in Theorem~\ref{tedec1}.	
\end{theorem}
\begin{proof}
    The total number of function evaluations depends on the maximum number of iterations and on the number of times the regularization parameter $\sigma_k$ is updated, since each unsuccessful attempt in Step~4 requires returning to Step~1.

    Let $r$ denote the total number of times $\sigma_k$ is increased throughout the execution of the algorithm. Since $\sigma_k$ is doubled whenever Step~4 fails and remains unchanged after a successful iteration, we have
$$
2^r\sigma_0\leq\sigma_{\max}.
$$
Consequently,
$$
r\leq\log_2\left(\frac{\sigma_{\max}}{\sigma_0}\right).
$$
On the other hand, if Algorithm~\ref{algfree} does not stop at iterations $0,1,\ldots,k$, then, as in the proof of Theorem~\ref{complexity}, and since $x^k\in\mathcal{C}$, Assumption~\ref{A:Fbounded} gives
$$
F^{\mathrm{low}}_{j_0}\leq F_{j_0}(x^k)\leq F_{j_0}(x^0)-k\,\dfrac{\alpha}{18c(1+\overline{B})^2}\,\epsilon^{\tfrac{\beta+1}{\beta}}.
$$
Hence the algorithm stops after finitely many iterations, and the maximum number of iterations at which it does not stop is bounded above by
$$
K_\epsilon=1+\left(F_{j_0}(x^0)-F^{\mathrm{low}}_{j_0}\right)
\frac{18c(1+\overline B)^2}{\alpha}
\epsilon^{-\frac{\beta+1}{\beta}}.
$$
Since each update of $\sigma_k$ requires an
additional execution of the algorithm without
increasing the iteration counter, by combining
the two bounds above, we obtain that
the maximum number of executions of
Steps 1 through 4 is upper-bounded by
$$
K_\epsilon+
\log_2\left(\frac{\sigma_{\max}}{\sigma_0}\right)+1.
$$
We now count the function evaluations requi per attempt. Among the finite-difference schemes conside, central differences require the largest number of function evaluations: each approximation $g_{f_j}(x^k,\lambda_k)$ requires $2n$ evaluations of $f_j$. Since $m=|\mathcal J|$, Step~1 requires at most $2mn$ function evaluations. Step~4 requires at most $m$ additional evaluations of the objective components at the trial point, reusing the values at $x^k$.

Thus, each execution requires at most
$m(2n+1)$ function evaluations. Therefore,
the total number of function evaluations
is bounded above by
$$
N_f\leq m(2n+1)\left[
K_\epsilon+
\log_2\left(\frac{\sigma_{\max}}{\sigma_0}\right)+1
\right].
$$
This completes the proof.

\end{proof}
\begin{remark}
The bound in Theorem~\ref{te.3} is obtained
by considering central finite differences.
When forward or backward finite differences
are employed, the factor $2n+1$ can be
replaced by $n+1$. On the other hand, when exact gradients are
used, they need to be computed only once
per iteration, regardless of the number of
updates of the regularization parameter
$\sigma_k$. Consequently, the number of
gradient evaluations is bounded above by
$m(K_\epsilon+1)$, independently of the
number of updates of $\sigma_k$.
\end{remark}

It is relevant to compare Theorems~\ref{complexity} and~\ref{te.3} with other complexity results in the literature, in which  H\"{o}lder continuity is imposed on the gradients. In this case, the complexity obtained in Theorem~\ref{complexity} is $\mathcal{O}(\epsilon^{-\frac{\beta+1}{\beta}})$ for the number of iterations under suitable parameter conditions, as specified in item~(b) of Proposition~\ref{prop:epspareto}. This is the same complexity as the method presented in \cite{Lizet} for $p=1$ (although this method was also designed for constrained problems), as shown in \cite[Theorem~2.3]{Lizet}. When $m=1$, we recover the complexity established in the theory for models with quadratic regularizations and H\"{o}lder conditions imposed on the gradient of the objective function, see for example \cite{Martinez2017, Yashtini2016}. In summary, in this work we obtain complexity results compatible with those established in both multiobjective optimization theory and scalar optimization.

\section{Applications and Numerical Results}\label{sec5}

In this section, we present numerical experiments to illustrate the applicability of Algorithm~\ref{algfree} (PDFPM) to multiobjective optimization problems. We are particularly interested in assessing the effectiveness of the algorithm when applied to composite problems \eqref{P1}, where $f_j: \mathbb{R}^n \to \mathbb{R}$ exhibits H\"{o}lder regularity for some objective $j \in \J$, and $h_j:\mathbb{R}^n\to\mathbb{R}\cup\{+\infty\}$
is given by $h_j=\widetilde h_j+\delta_{\mathcal{C}}$,
where $\widetilde h_j:\mathbb{R}^n\to\mathbb{R}$ is convex
and finite-valued, but not necessarily differentiable,
and $\delta_{\mathcal{C}}$ is the indicator function of a nonempty,
closed, convex set $\mathcal{C}\subseteq\mathbb{R}^n$. A direct application of composite problems is robust multiobjective optimization, which will be discussed in Subsection~\ref{subsec_robust_aplic}. For further details on the wide range of applications involving the minimization of composite problems, see \cite{fukuda2019, Chen2019} and the references therein.

The numerical experiments, which will be presented in Subsection~\ref{sub_sec_numerical}, were conducted using the Julia programming language (version 1.12.1) on a machine equipped with two Intel Xeon X5675 processors (3.07~GHz, 12 cores) and 188 GB of RAM. We used standard packages for nonlinear and linear optimization. The source code is available at \url{https://github.com/Souza-DR/AAS2025-PDFreeMO}.

\subsection{Application to Robust Multiobjective Optimization}\label{subsec_robust_aplic}

Let $f_j: \mathbb{R}^n \to \mathbb{R}$ be differentiable for all $j \in \mathcal{J}$. The multiobjective optimization problem with uncertainty is defined as follows:
\begin{equation}\label{uncertain}
    \min_{x \in \R^n} (f_1(x) + \hat{h}_1(x, z), \ldots, f_m(x) + \hat{h}_m(x, z)),
\end{equation}
where, for each $j \in \mathcal{J}$, $z \in \mathcal{Z}_j$ is a deterministic uncertain parameter, $\mathcal{Z}_j$ is an uncertainty set that can be estimated, and $\hat{h}_j : \mathbb{R}^n\times\mathbb{R}^n \to \mathbb{R} \cup \{+\infty\}$ is convex with respect to the first argument. The robust version of problem \eqref{uncertain} is given by:
\begin{equation}\label{robust:prob}
     \min_{x \in \R^n} (f_1(x) + \max_{z\in \mathcal{Z}_1}\hat{h}_1(x, z), \ldots, f_m(x) + \max_{z\in \mathcal{Z}_m}\hat{h}_m(x, z)).
\end{equation}
If we define $\widetilde{h}_j(x) \coloneqq \max_{z\in \mathcal{Z}_j}\hat{h}_j(x, z)$, assumed finite, for all $j \in \mathcal{J}$, then problem \eqref{P1} becomes a robust optimization problem in which the objectives are affected by uncertainties. For further details on the robust formulation of uncertain problems, see \cite{Chen2019, fukuda2019} and the references therein.

In the proof of Lemma~\ref{BD}, we saw that the solution to the subproblem \eqref{probminmax} is equivalent to solving the following constrained optimization problem:
\begin{equation}\label{probztau}
\begin{aligned}
\min_{x \in \mathbb{R}^n,\, \tau \in \mathbb{R}} \quad & \tau + \frac{\sigma_k}{2} \|x - x^k\|^2 \\
\text{s.t.} \quad & \langle g_{f_j}(x^k, \lambda_k), x - x^k \rangle 
+ \frac{1}{2} \langle B_j^k (x - x^k), x - x^k \rangle \\
& + h_j(x) - h_j(x^k) \leq \tau, \quad \forall j \in \mathcal{J}.
\end{aligned}
\end{equation}
Note that, for $j \in \mathcal{J}$, even if $\hat{h}_j(\cdot,\cdot)$ is differentiable, $h_j(\cdot)$  is not necessarily differentiable. In general, the term $h_j(\cdot)$ is not easily computable, which makes solving subproblem \eqref{probztau} particularly challenging. However, when $\hat{h}_j(\cdot,\cdot)$ and $\mathcal{Z}_j$ have a special structure, it is possible to simplify the resolution of subproblem \eqref{probztau} by applying duality theory.

According to the framework presented in \cite{fukuda2019}, for each $j \in \mathcal{J}$, suppose that $\hat{h}_j(x, z) \coloneqq \langle x, z \rangle$ (linear with respect to the first argument) and $\mathcal{Z}_j = \{ z \in \mathbb{R}^n \mid A_j z \leq b_j \}$ (a polyhedron), where $A_j \in \mathbb{R}^{d \times n}$ and $b_j \in \mathbb{R}^{d}$. In this case, we can rewrite $\widetilde{h}_j(x) = \max_{z \in \mathcal{Z}_j} \langle x, z \rangle$ as the following linear programming problem:
\begin{equation*}
\begin{aligned}
\max_{z} \quad & \langle x, z \rangle \\
\text{s.t.} \quad & A_j z \leq b_j.
\end{aligned}
\end{equation*}
Its dual problem is given by
\begin{equation*}
\begin{aligned}
\min_{w} \quad & \langle b_j, w \rangle \\
\text{s.t.} \quad & A_j^\top w = x, \\
& w \geq 0.
\end{aligned}
\end{equation*}

By applying strong duality theory, we conclude that solving problem \eqref{probminmax} is equivalent to solving problem \eqref{probztau}, which in turn can be reformulated as the following convex quadratically constrained quadratic programming (QCQP) problem:
\begin{equation}\label{probztaudual}
\begin{aligned}
\min_{x \in {\mathcal{C}},\, \tau \in \mathbb{R},\, w_j} \quad & \tau + \frac{\sigma_k}{2} \|x - x^k\|^2 \\
\text{s.t.} \quad 
& \langle g_{f_j}(x^k, \lambda_k), x - x^k \rangle 
+ \frac{1}{2} \langle B_j^k(x - x^k), x - x^k \rangle \\
& + \langle b_j, w_j \rangle - h_j(x^k) \leq \tau, \quad \forall j \in \mathcal{J}, \\
& A_j^\top w_j = x, \quad \forall j \in \mathcal{J}, \\
& w_j \geq 0, \quad \forall j \in \mathcal{J}.
\end{aligned}
\end{equation}
Thus, for specific structures of $h_j(\cdot)$ and $\mathcal{Z}_j$, the subproblem arising in Algorithm~\ref{algfree} becomes tractable. Moreover, this reformulation highlights how duality theory provides not only a theoretical characterization but also a practical computational advantage, since convex QCQP is a well-studied class of problems for which efficient solvers are available. This observation reinforces the applicability of our method in robust multiobjective optimization settings, closing the gap between abstract formulations and implementable algorithms.

\subsection{Numerical Results}\label{sub_sec_numerical}
In the tests conducted with the PDFPM method, we used the following parameters:
$\alpha = 0.1$, $\epsilon = 10^{-4}$, $\sigma_0 = 1.0$, and $B_j^0 = I_n$ for all
$j \in \mathcal{J}$. The implementation also stops PDFPM when $20$ consecutive trial points $\bar{x}^k$ are rejected in Step~4 within one iteration. We also employed quasi-Newton matrices $B_j^k$ updated by
BFGS-type schemes, one matrix per objective function. In the scalar case, it is well known (see \cite{nocedal}) that,
given a symmetric matrix $B_j^k > 0$, the classical BFGS update preserves positive
definiteness, i.e., $B_j^{k+1} > 0$, provided that the curvature condition
$(s^k)^\top y_j^k > 0$ holds, where
$s^k \coloneqq x^{k+1} - x^k$ and
$y_j^k \coloneqq \nabla f_j(x^{k+1}) - \nabla f_j(x^k)$ for all $j \in \mathcal{J}$.
Such a condition can be enforced, for instance, if $f_j$ is strictly convex or if the
step length satisfies Wolfe conditions. In the multiobjective setting there is no
exact counterpart: if all components $f_j$ are strictly convex, then
$(s^k)^\top y_j^k > 0$ holds for each $j$, but without convexity assumptions, even a
multiobjective Wolfe line search \cite{PerezPrudente2019} does not guarantee that
the curvature condition is satisfied for every objective. Motivated by this difficulty,
\cite{Prudente2022} proposed a modified BFGS-type update that preserves positive
definiteness when the step length is obtained by a Wolfe line search, without assuming
convexity.

Since Algorithm~\ref{algfree} does not employ a line search, we combine this modified
BFGS update with a cautious strategy in order to maintain $B_j^k$ positive
definite. For each objective $j \in \mathcal{J}$, if the curvature condition
$(s^k)^\top y_j^k > \xi_k$ holds, we perform the classical BFGS update
\[
B_j^{k+1} \coloneqq
B_j^{k}
-\frac{B_j^{k} s^k (s^k)^\top B_j^{k}}{(s^k)^\top B_j^{k} s^k}
+\frac{y_j^{k} (y_j^{k})^\top}{(y_j^{k})^\top s^k}.
\]
If $(s^k)^\top y_j^k \le {\xi_k}$ but
\[
\rho_j^k \coloneqq
\max_{\ell \in \mathcal{J}}\big\{\nabla f_\ell(x^{k+1})^\top s^k\big\}
- \nabla f_j(x^k)^\top s^k > \xi_k,
\]
then we update $B_j^k$ by the modified BFGS rule
\begin{align*}
	\begin{split}
		B_j^{k+1} ~\coloneqq ~
		&  B_j^{k}
		- \frac{\rho_j^k\, B_j^{k} s^k (s^k)^\top B_j^{k}}
		{\big(\rho_j^k - (s^k)^\top y_j^{k}\big)^2
		 + \rho_j^k (s^k)^\top B_j^k s^k} \\
		& + \frac{\big((s^k)^\top B_j^{k} s^k\big)\, y_j^{k} (y_j^{k})^\top}
		{\big(\rho_j^k - (s^k)^\top y_j^{k}\big)^2
		 + \rho_j^k (s^k)^\top B_j^k s^k}  \\
		& + \big(\rho_j^k - (s^k)^\top y_j^{k}\big)
		\frac{y_j^{k} (s^k)^\top B_j^{k}
		      + B_j^{k} s^k (y_j^{k})^\top}
		{\big(\rho_j^k - (s^k)^\top y_j^{k}\big)^2
		 + \rho_j^k (s^k)^\top B_j^k s^k}.
	\end{split}
\end{align*}
If neither $(s^k)^\top y_j^k > \xi_k$ nor $\rho_j^k > \xi_k$ holds, we simply set
$B_j^{k+1} \coloneqq B_j^k$, i.e., no update is performed. Here $\xi_k \coloneqq 10^{-6}\min\{|\vartheta_k|,1\}$, where $\vartheta_k$ is the optimal value of \eqref{probminmax} at $\bar{x}^k=x^{k+1}$. It is worth noting that the vectors $y_j^k$ are computed from the gradients of
$f_j$ in the composite problem~\eqref{P1}, and whenever $\nabla f_j$ is not
available in closed form these gradients are approximated by finite differences,
for all $j \in \mathcal{J}$, see Remark~\ref{diff}. In all experiments reported in this section, the gradients of $f_j$ are available in closed form, provided by the package \texttt{MultiObjectiveProblems.jl} \cite{Souza_MultiObjectiveProblems_jl}; except in the comparison of approximate Pareto fronts at the end of Subsection~\ref{subsec_comparison}, we take $\lambda_k=0$ in Step~1, so that $g_{f_j}(x^k,\lambda_k)=\nabla f_j(x^k)$, and no finite differences are used. In that comparison, PDFPM is also run with the central differences of Remark~\ref{diff}(iii) and $\lambda_k=\epsilon/(\sigma_k\sqrt{n})$, the largest value allowed in Step~1; each approximation of the gradients costs $2n$ evaluations of $f=(f_1,\ldots,f_m)$, which are counted as function evaluations.

Considering the discussion in Subsection~\ref{subsec_robust_aplic}, we now apply the PDFPM method to a robust optimization problem. For each test problem and each objective $j$, we generate a matrix $\Tilde{A}_j \in \R^{n \times n}$ with entries sampled uniformly in $[0,1]$, conditioned on $\|\Tilde{A}_j\|_2\|\Tilde{A}_j^{-1}\|_2\le 10^4$, and keep it fixed across all uncertainty levels. Then, for a given $\delta \ge 0$, we define
\begin{equation}\label{maxh}
\widetilde{h}_j(x) = \max_{z \in \mathcal{Z}_j} \langle x, z \rangle \quad \text{and} \quad \mathcal{Z}_j = \{ z \in \R^n \mid -\delta e \leq \Tilde{A}_j z \leq \delta e \},
\end{equation}
where $e = (1, \ldots, 1)^\top \in \mathbb{R}^{n}$. By taking $A_j \coloneqq [\Tilde{A}_j; - \Tilde{A}_j] \in \mathbb{R}^{2n \times n}$ and $b_j \coloneqq (\delta, \ldots, \delta)^\top \in \mathbb{R}^{2n}$, we obtain a robust optimization problem as defined in \eqref{robust:prob}, where $\delta$ controls the level of uncertainty in the problem. In particular, $\delta = 0$ corresponds to the nominal (non-robust) case without uncertainty.

The presence of the pointwise maximization in \eqref{maxh} makes the function $h_j$, for each $j \in \mathcal{J}$, convex but not necessarily differentiable. Since $\Tilde{A}_j$ is nonsingular, the change of variables $u=\Tilde{A}_j z$ in \eqref{maxh} gives $\widetilde{h}_j(x)=\delta\|\Tilde{A}_j^{-\top}x\|_1$, which we use to evaluate $h_j$. The bound $\|\Tilde{A}_j\|_2\|\Tilde{A}_j^{-1}\|_2\le 10^4$ keeps the relative rounding error of this formula of order $10^4\mathbf{u}\approx 1.1\times 10^{-12}$, where $\mathbf{u}$ is the unit roundoff, below the tolerance $10^{-10}$ used for the subproblems. At each iteration of the algorithm, we solve a single convex QCQP of the form \eqref{probztaudual}. We solve it with Ipopt~\cite{Ipopt}, an interior-point optimizer that efficiently handles nonlinear constraints through a primal–dual method with a filter line-search approach.

To assess the points produced by the methods, we do not rely on their termination tests, which measure different quantities: Step~3 of Algorithm~\ref{algfree} tests the length of the proximal step, whereas the methods of Subsection~\ref{subsec_comparison} test the optimal values of their own subproblems. Instead, we use two merit functions of \cite{TanabeOptLett2023,TanabeMerit2024}, which depend only on the problem and not on the method that produced the point:
\begin{align}
u_0(x) &\coloneqq \sup_{y\in\R^n}\min_{j\in\mathcal{J}}\big[F_j(x)-F_j(y)\big],\label{eq:u0}\\
w_1(x) &\coloneqq \max_{y\in\R^n}\min_{j\in\mathcal{J}}\Big[\langle\nabla f_j(x),x-y\rangle+h_j(x)-h_j(y)-\frac{1}{2}\|x-y\|^2\Big],\label{eq:w1}
\end{align}
Both are nonnegative; $u_0$ vanishes exactly at the weakly Pareto optimal points and $w_1$ exactly at the Pareto critical points \cite[Theorem~3.1]{TanabeOptLett2023}. Evaluating $u_0$ requires a global minimization over $y$, which is a convex problem when the $F_j$ are convex, whereas $w_1$ is the optimal value of a strongly convex problem \cite[Remark~5.2]{TanabeOptLett2023}. Accordingly, Subsection~\ref{holderset}, devoted to two convex problems, uses $u_0$, and Subsection~\ref{subsec_comparison}, whose test set contains nonconvex problems, uses $w_1$.

\subsubsection{H\"{o}lder Setting}\label{holderset}
To illustrate the applicability of the proposed algorithm, we define two biobjective functions with the following properties: the first vector function has one objective with a Lipschitz continuous gradient, while the other has a $(p-1)$-H\"{o}lder continuous gradient; the second vector function has both objectives with $(p-1)$-H\"{o}lder continuous gradients. See Examples~\ref{ex_1} and~\ref{ex_2} below. Note that, for $1 < p < 2$, the mapping $x \mapsto \|x\|_p^p$ has a gradient that is $(p-1)$-H\"older continuous. Consequently, in Example~\ref{ex_1} the gradient of $f_2$ is $(p-1)$-H\"older continuous, whereas in Example~\ref{ex_2} the gradients of $f_1$ and $f_2$ are $(p_1-1)$- and $(p_2-1)$-H\"older continuous, respectively. Hence both AAS1 and AAS2 satisfy the H\"older-type assumptions adopted in this work.

\begin{example}[AAS1]\label{ex_1}
    Let $F: \R^2 \to \R^2$ be defined by $F(x) = (f_1(x), f_2(x))$ with
\begin{align*}
f_1(x) = \frac{1}{2}\|A x - b\|_2^2, \quad f_2(x) = \frac{\mu}{p}\|D x\|_p^p,
\end{align*}
where $1 < p < 2$. The first objective exhibits a quadratic structure with strong convexity, while the second objective represents a $p$-norm regularization term with H\"{o}lder continuous gradients. The problem parameters are specified as follows: 
\[
A = \begin{bmatrix} 2.0 & 0.5 \\ 0.5 & 1.5 \end{bmatrix}, \quad b = [1.0, -0.5]^\top, \quad p = 1.003, \quad \mu = 0.9, \quad D = \begin{bmatrix} 1.0 & 0.8 \\ 0.3 & 1.2 \end{bmatrix}.
\]
\end{example}

\begin{example}[AAS2]\label{ex_2}
    Consider the following biobjective function:
\begin{align*}
f_1(x) = \frac{\mu_1}{p_1}\|D_1(x - c_1)\|_{p_1}^{p_1}, \quad f_2(x) = \frac{\mu_2}{p_2}\|D_2(x - c_2)\|_{p_2}^{p_2},
\end{align*}
where $1 < p_1, p_2 < 2$. Both objective functions are convex and possess H\"{o}lder continuous gradients. The problem parameters are specified as follows: 
\[
p_1 = 1.003, \quad \mu_1 = 1.2, \quad D_1 = \begin{bmatrix} 1.2 & -0.3 \\ 0.4 & 1.5 \end{bmatrix}, \quad c_1 = [1.5, -1.0]^\top,
\]
\[
p_2 = 1.002, \quad \mu_2 = 0.8, \quad D_2 = \begin{bmatrix} 1.8 & 0.5 \\ -0.2 & 1.1 \end{bmatrix}, \quad c_2 = [-1.2, 0.8]^\top.
\]
\end{example}

In the experiments, we considered different uncertainty scenarios in the minimization of the biobjective functions AAS1 and AAS2, namely $\delta \in \{0.0,~0.02,~0.05,~0.1\}$. For each fixed value of $\delta$, we generated $m$ matrices $A_j \in \mathbb{R}^{2n \times n}$, with $j \in \mathcal{J}$, thereby defining our robust problem~\eqref{robust:prob}. For each problem instance, the algorithm was executed $200$ times with randomly selected starting points $x^0 \in {\mathcal{C}}$, where the box in the definition of $h_j$ is $\mathcal{C} \coloneqq [-2,2] \times [-2,2]$ for AAS1 and $\mathcal{C} \coloneqq [-5,5] \times [-5,5]$ for AAS2.

Figures~\ref{fig:pareto-AAS1} and~\ref{fig:pareto-AAS2} correspond to the case $\delta = 0$, that is, a scenario without uncertainty. The right panel illustrates the convergence behavior of the algorithm, highlighting its ability to approximate the true Pareto front. We also plot the initial point (in gray) and the final point (in red) of the optimization for each of the $200$ runs, successful or not, connected by a dashed line. In addition, the left panel in both figures shows the complete set of objective function values obtained by evaluating the objectives on a uniform grid of points, thus providing a comprehensive view of the objective space landscape. Figure~\ref{fig:delta_comparison} presents the comparison of the reconstructions obtained for the other values of $\delta$. On the left, the approximations of the Pareto fronts obtained for AAS1 are shown, and on the right, those for AAS2. One observes that as the level of uncertainty increases, the approximation of the Pareto front produced by PDFPM deviates further from the expected front.

\begin{figure}[h]
\centering
\begin{minipage}{0.48\textwidth}
\centering
\includegraphics[width=\textwidth]{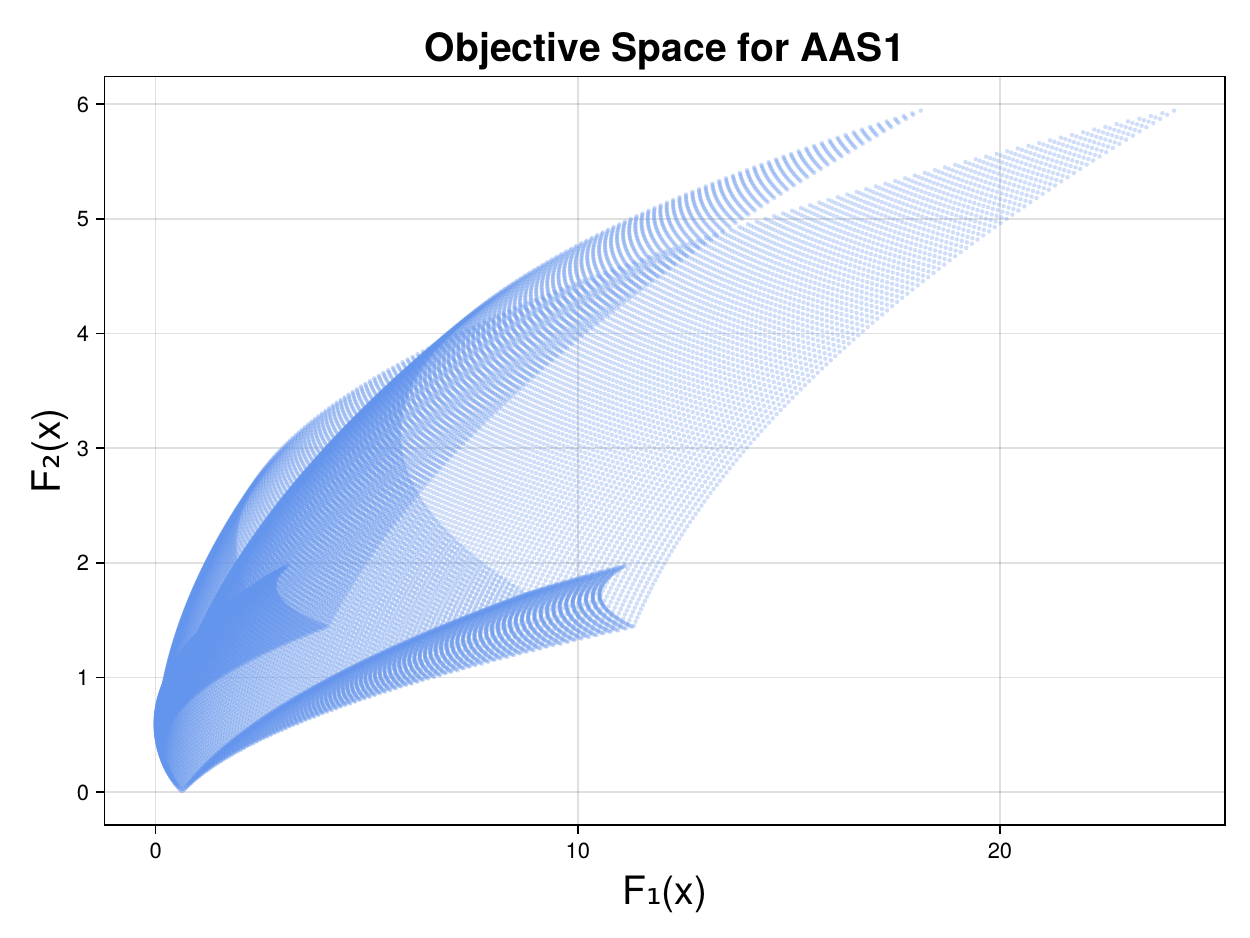}
\end{minipage}
\hfill
\begin{minipage}{0.48\textwidth}
\centering
\includegraphics[width=\textwidth]{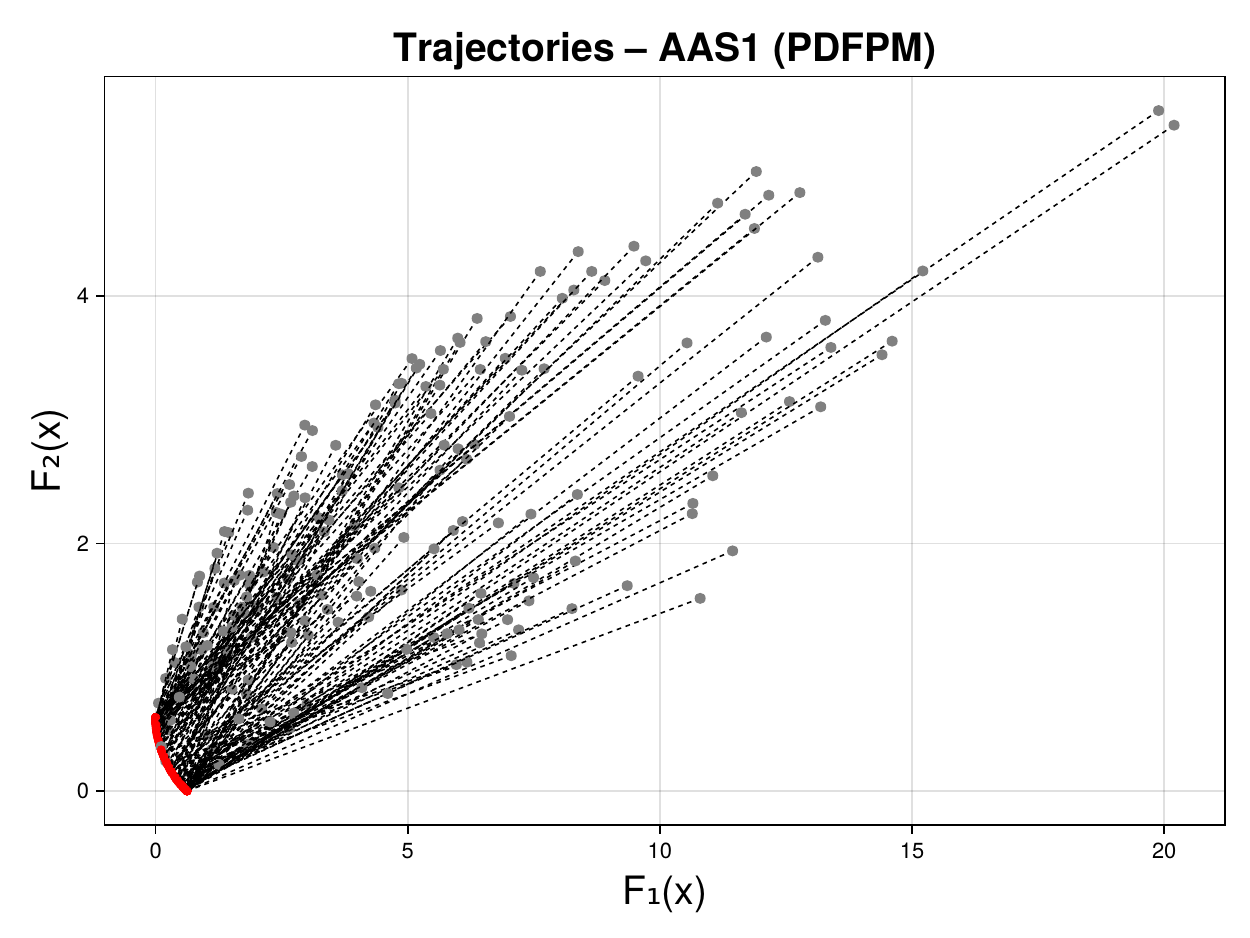}
\end{minipage}
\caption{Objective space and approximated Pareto front obtained by the PDFPM for the AAS1.}
\label{fig:pareto-AAS1}
\end{figure}

\begin{figure}[h]
\centering
\begin{minipage}{0.48\textwidth}
\centering
\includegraphics[width=\textwidth]{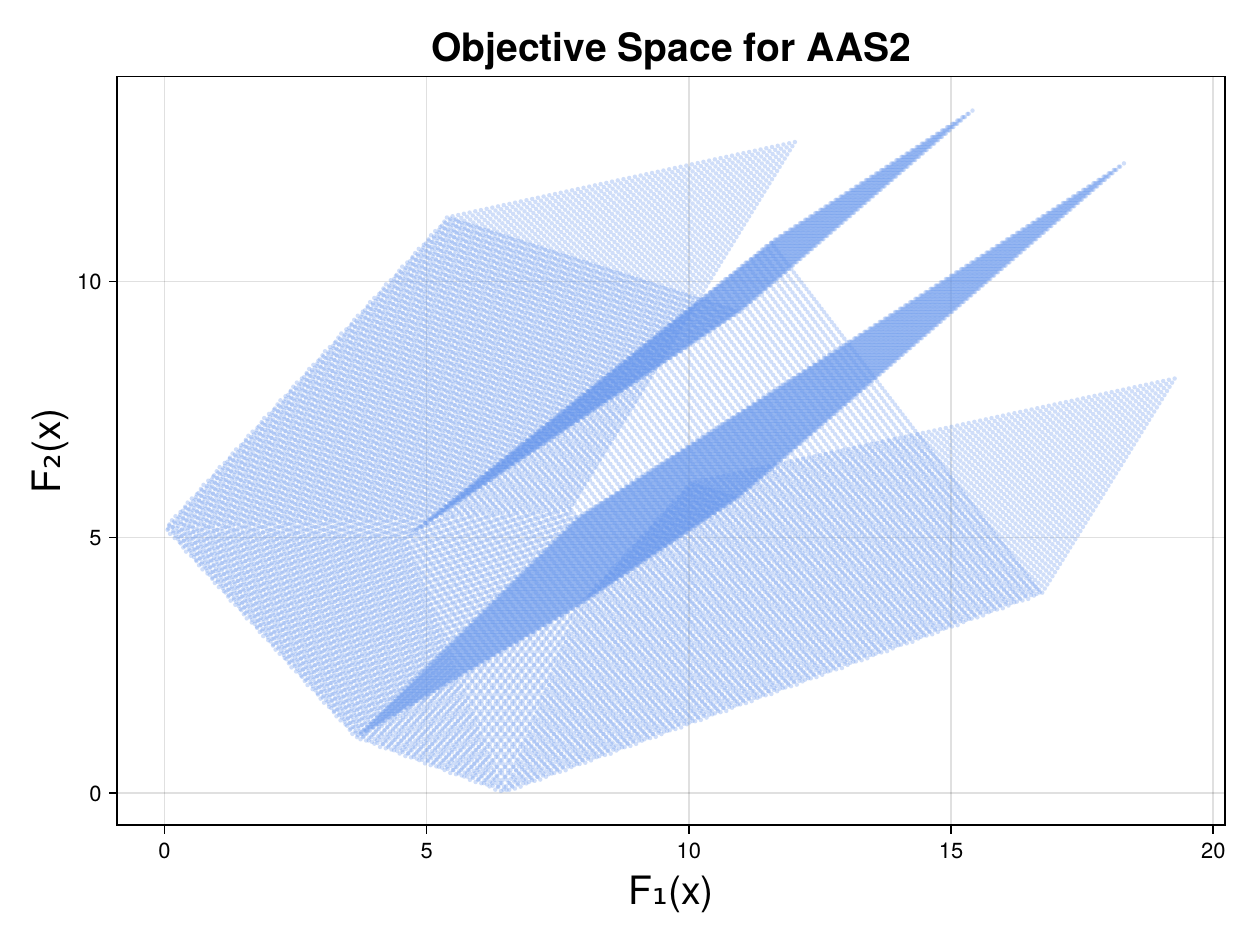}
\end{minipage}
\hfill
\begin{minipage}{0.48\textwidth}
\centering
\includegraphics[width=\textwidth]{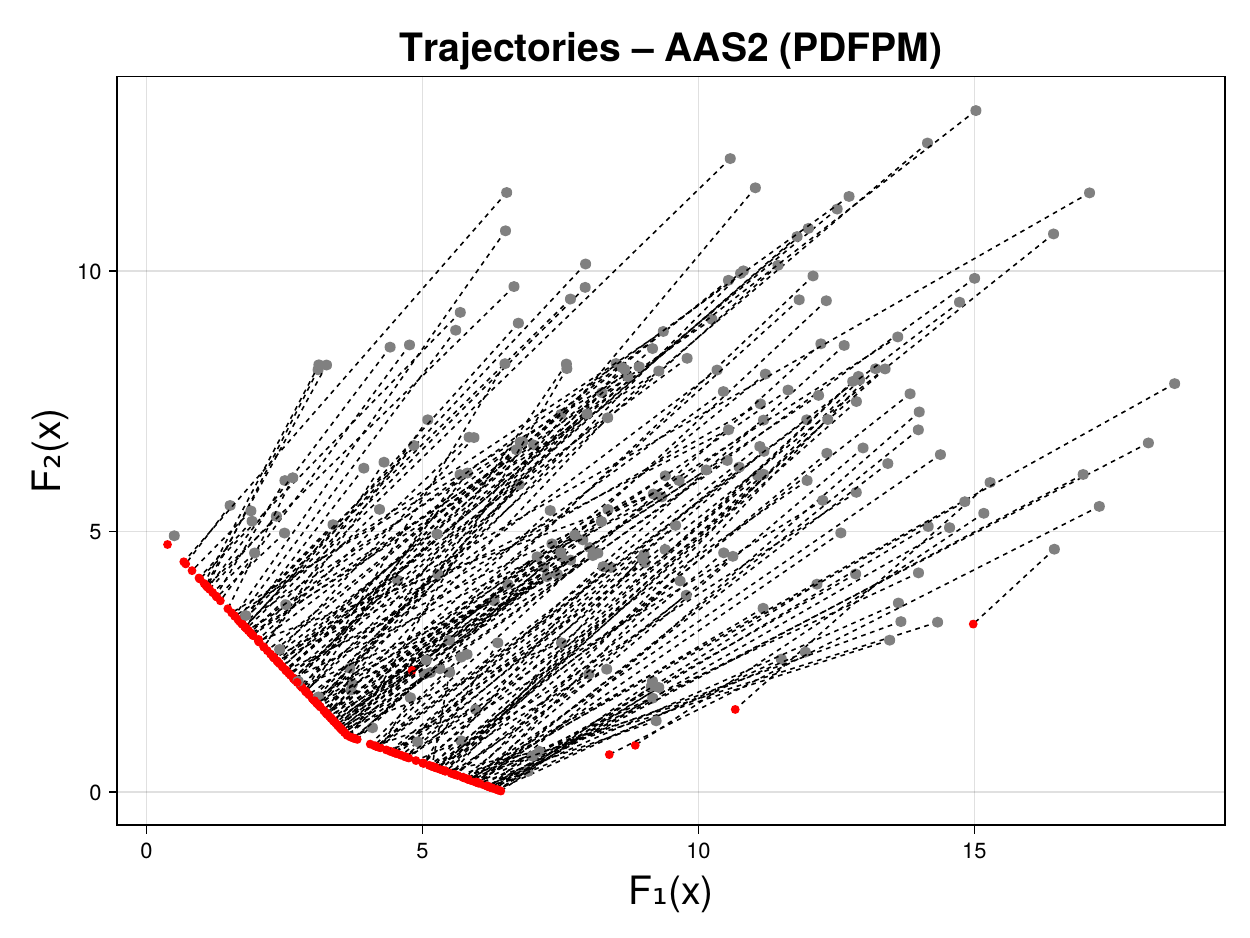}
\end{minipage}
\caption{Objective space and approximated Pareto front obtained by the PDFPM for the AAS2.}
\label{fig:pareto-AAS2}
\end{figure}

\begin{figure}[h]
\centering
\begin{minipage}{0.48\textwidth}
    \centering
    \includegraphics[width=\textwidth]{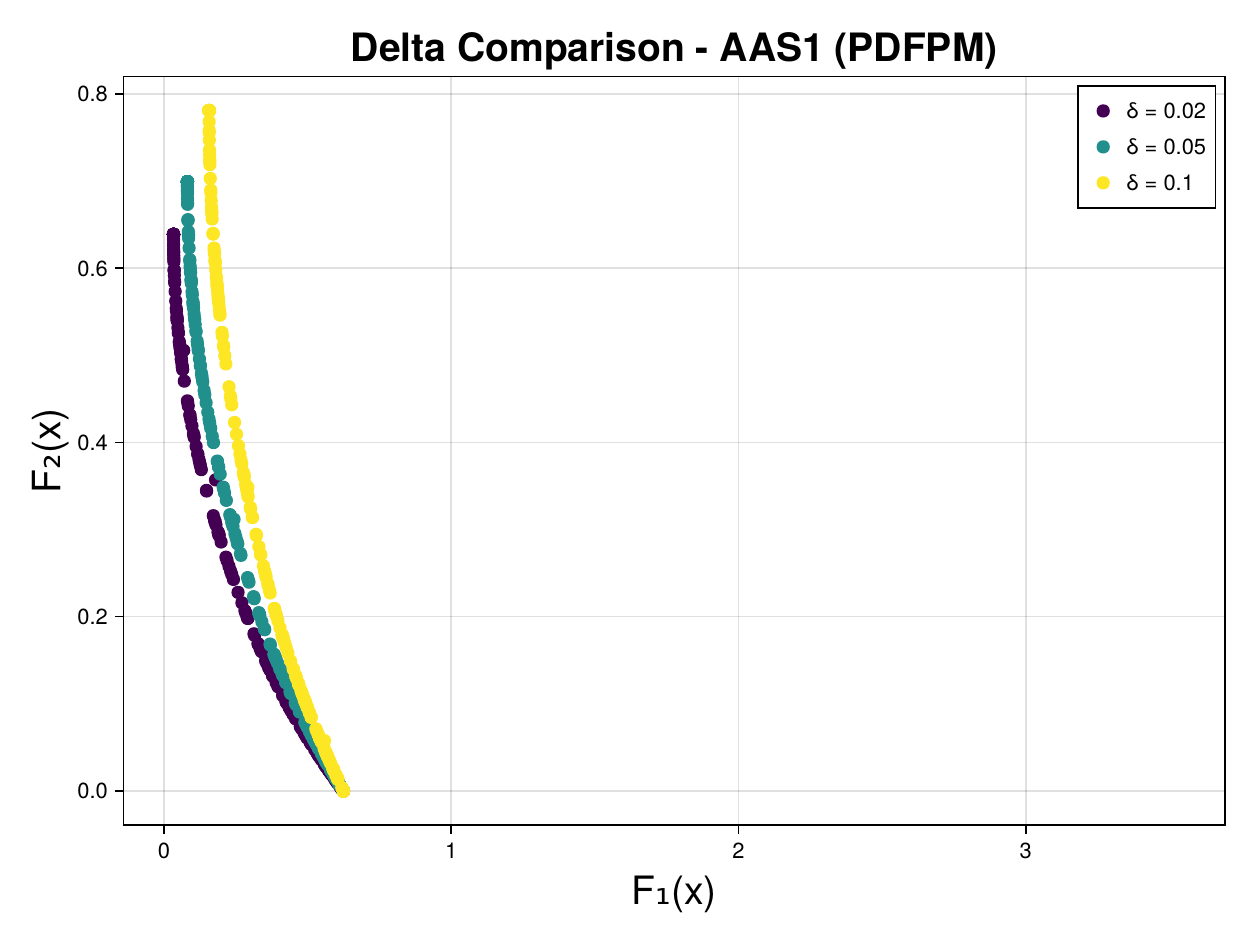}
\end{minipage}
\hfill
\begin{minipage}{0.48\textwidth}
    \centering
    \includegraphics[width=\textwidth]{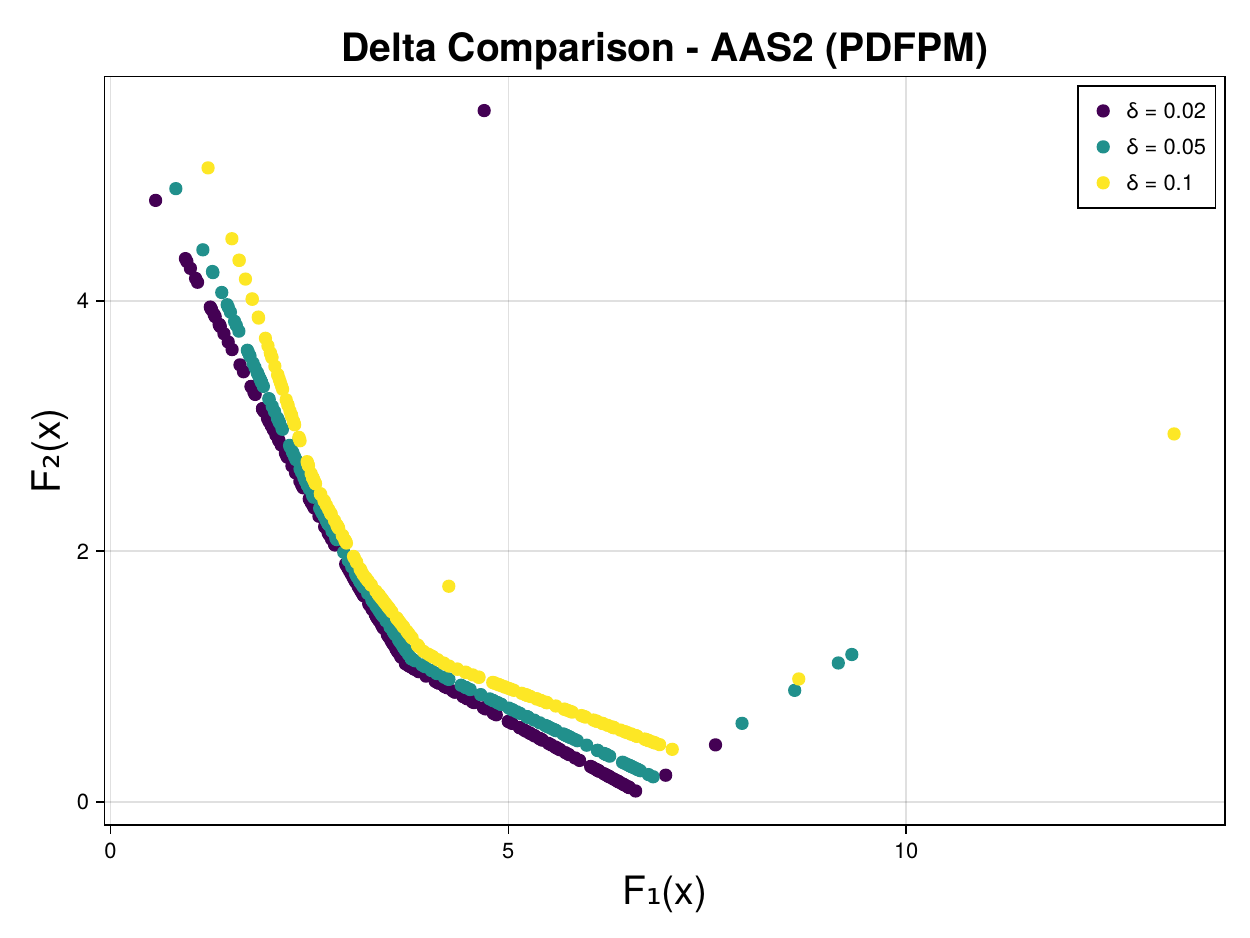}
\end{minipage}
\caption{Comparison between the complete objective space and the Pareto front approximations obtained by PDFPM for the AAS1 and AAS2 problems under different uncertainty levels.}
\label{fig:delta_comparison}
\end{figure}

For each problem and each value of $\delta$, we assess by $u_0$ the point returned by each of the $200$ runs, whatever the reason for termination; a run is successful if $u_0\le 10^{-8}$. Since AAS1 and AAS2 are convex, the zeros of $u_0$ are the Pareto critical points sought by Algorithm~\ref{algfree}. A measure based on function values is needed here: the exponents $\beta_j=p_j-1\le 0.003$ make $\nabla f_j$ almost discontinuous, and the factor $\big(\epsilon/(3\sigma_k\theta_k)\big)^{\beta_j}$ in \eqref{eq:pareto-residual} equals $0.97$ for $\epsilon=10^{-4}$, $\sigma_k\theta_k=1$ and $\beta_j=0.003$, so the test of Step~3, like any gradient-based measure such as $w_1$, conveys little information on the returned point. For this reason, AAS1 and AAS2 are not included in the comparison of Subsection~\ref{subsec_comparison}.

Table~\ref{tab:pdfpm-results} reports the results. The failures come from runs stopped by the iteration limit (one, on AAS1, by the safeguard of $20$ rejected trial points), except for $4$ points for which the numerical computation of $u_0$ does not decide whether $u_0\le 10^{-8}$. These results, with about $97\%$ of the returned points weakly Pareto optimal up to $u_0\le 10^{-8}$ for both problems, suggest that PDFPM handles the different uncertainty levels in a stable way on these problems with H\"older continuous gradients.

\begin{table}[h]
\centering
\caption{Number of returned points of PDFPM with $u_0\le 10^{-8}$ (200 runs per setting; points at which the computed bounds on $u_0$ do not decide are counted as failures).}
\label{tab:pdfpm-results}
\begin{tabular}{lcccc}
\hline
Problem & $\delta=0.0$ & $\delta=0.02$ & $\delta=0.05$ & $\delta=0.1$ \\
\hline
AAS1 & 194 (97\%) & 191 (95.5\%) & 198 (99\%) & 192 (96\%) \\
AAS2 & 193 (96.5\%) & 196 (98\%) & 192 (96\%) & 195 (97.5\%) \\
\hline
\end{tabular}
\end{table}

\subsubsection{Comparison with ProxGrad and CondG}\label{subsec_comparison}

To assess the performance of Algorithm~\ref{algfree} on standard multiobjective test problems, we compare PDFPM with two first-order methods designed for composite multiobjective optimization. The first one is the proximal gradient method proposed in \cite[Algorithm~3.1]{fukuda2019}, which we refer to as ProxGrad. The second is the conditional gradient method introduced in \cite[Algorithm~1]{Assuncao2023}, denoted by CondG. Although ProxGrad, CondG, and Algorithm~\ref{algfree} are all tailored to composite multiobjective problems, there are important differences in the underlying subproblems, in the corresponding search directions, and in the stopping criteria.

As discussed in Subsection~\ref{subsec_robust_aplic}, Step~2 of Algorithm~\ref{algfree} requires the solution of the subproblem \eqref{probminmax}, which is equivalent to \eqref{probztau}. In an analogous way, ProxGrad requires the solution of the following subproblem
\begin{equation}\label{ProxGrad_subprob}
    \min_{x\in\R^n}\;
    \max_{j\in\mathcal{J}}
    \left[
      \langle\nabla f_j(x^k),x-x^k\rangle
      + h_j(x) - h_j(x^k)
    \right]
    + \frac{\ell}{2}\,\|x - x^k\|^2,
\end{equation}
where $\ell > 0$ is an algorithmic parameter. In our implementation the solution of \eqref{ProxGrad_subprob}, denoted by $\bar{x}^{k}_{pg}$, is computed with Ipopt~\cite{Ipopt}. The corresponding search direction is then given by \(d(x^k) = \bar{x}^{k}_{pg} - x^k,\)
and we denote the optimal value of \eqref{ProxGrad_subprob} at $\bar{x}^{k}_{pg}$ by $\theta_{pg}(x^k)$.

For CondG the subproblem is
\begin{equation}\label{CondG_subprob}
    \min_{x\in\R^n}\;
    \max_{j\in\mathcal{J}}
    \left[
      \langle\nabla f_j(x^k),x-x^k\rangle
      + h_j(x) - h_j(x^k)
    \right],
\end{equation}
and the search direction is defined as
\(d(x^k) = \bar{x}^{k}_{cg} - x^k,\)
where $\bar{x}^{k}_{cg}$ is a solution of \eqref{CondG_subprob}. The optimal value of \eqref{CondG_subprob} at $\bar{x}^{k}_{cg}$ is denoted by $\theta_{cg}(x^k)$. In our experiments, \eqref{CondG_subprob} is solved with HiGHS~\cite{HiGHS}. For both ProxGrad and CondG, a line search is performed along the direction $d(x^k)$. If $x^k$ is noncritical and $d(x^k)$ is a descent direction, we look for $\eta_k \in (0,1]$ such that
\[
    F_j\bigl(x^k + \eta_k\, d(x^k)\bigr) 
    \leq 
    F_j(x^k) + \zeta\, \eta_k\, \theta(x^k)
    \qquad \forall j \in \mathcal{J},
\]
where $\zeta = 10^{-4}$ and $\theta(x^k)$ denotes the optimal value of the corresponding subproblem (that is, $\theta_{pg}$ for ProxGrad and $\theta_{cg}$ for CondG). After the line search, the new iterate is given by $x^{k+1} = x^k + \eta_k\, d(x^k)$.

The stopping tests proposed for ProxGrad and CondG are $\theta_{pg}(x^k)=0$ and $\theta_{cg}(x^k)=0$, respectively. Since the termination tests of the three methods measure different quantities, we compare them through the merit function $w_1$ in \eqref{eq:w1}. Each method is run for at most $100$ iterations, and $w_1$ is evaluated at every iterate. A run is successful if some iterate $x^k$ satisfies $w_1(x^k)\le 10^{-6}$, and its performance measures (iterations, evaluations and time) are those spent up to the first such iterate, so that the three methods are measured at the same accuracy; the evaluations of $w_1$ are not counted. ProxGrad and CondG run until the iteration limit, until they cannot proceed, for instance when the line search fails, until $\theta_{cg}(x^k)=0$ for CondG, or, to save time, until they have met their own tests with tolerance $10^{-4}$ and attained $w_1\le 5\times10^{-9}$, which does not affect the measures. PDFPM keeps its termination at Step~3 and the safeguard of $20$ rejected trial points. A run that ends before attaining $w_1\le 10^{-6}$ is unsuccessful. On the problems of Table~\ref{tab:problems}, every run of PDFPM that terminated at Step~3 had already attained $w_1\le 10^{-6}$, so this termination does not affect the comparison. The threshold $10^{-6}$ is absolute, hence not invariant to the scaling of $F$. Instances whose starting point already satisfies $w_1(x^0)\le 10^{-6}$ are not considered. We use $\ell=1$ in \eqref{ProxGrad_subprob}, so that $w_1(x^k)=-\theta_{pg}(x^k)$ and, for ProxGrad, this criterion reads $|\theta_{pg}(x^k)|\le 10^{-6}$. The subproblems of the three methods are solved with tolerance $10^{-10}$.

The test problems are listed in Table~\ref{tab:problems} and are frequently used to validate multiobjective optimization methods, see for instance \cite{Assuncao2021, Prudente2022} and the references therein. As in Subsection~\ref{holderset}, we consider different uncertainty levels in the minimization of the functions in Table~\ref{tab:problems}, namely
\( \delta \in \{0.0, 0.02, 0.05, 0.1\}.\)
For each fixed value of $\delta$, we generate $m$ matrices $A_j \in \R^{2n \times n}$, with $j \in \mathcal{J}$, thus defining the robust problem~\eqref{robust:prob}, as discussed previously.
For each problem, $\mathcal{C}$ is a box $\{x\in\mathbb{R}^n: l\leq x\leq u\}$, given by the bounds registered for the problem in the package \texttt{MultiObjectiveProblems.jl} \cite{Souza_MultiObjectiveProblems_jl} or, when there are none, by the working box recommended by that package, and common to the three methods; its compactness is required by CondG \cite[Section~3]{Assuncao2023} and ensures that \eqref{CondG_subprob} has a solution.
 For each problem instance, the algorithm under consideration is executed $200$ times with randomly selected starting points in $\mathcal{C}$. The same collection of $200$ starting points is used for PDFPM, ProxGrad, and CondG to ensure a fair comparison.

\begin{table}[ht]
\caption{List of test problems used in the comparison.}
\label{tab:problems}
\scriptsize

\centering
\begin{tabular}{@{}c@{\hspace{0.02\textwidth}}c@{}}
\rowcolors{2}{}{lightgray}
\begin{tabular}{|l|l|c|c|}
\hline
\rowcolor[gray]{.90}
Problem & Source & $n$ & $m$  \\
\hline
AP1&  \cite{ansary} &     2 &  3 \\
 AP2&  \cite{ansary} &     1 &  2 \\
 AP3&  \cite{ansary} &     2 &  2 \\
 AP4&  \cite{ansary}  &     3 &  3 \\
 BK1&  \cite{reviewproblems}  &  2 &  2  \\ 
 DD1&  \cite{dd1} &  5 &  2 \\
 DGO0 & \cite{DGO2000} & 1 & 2 \\
 DGO1&  \cite{DGO2000}  &  1 &  2 \\ 
 DGO2&  \cite{DGO2000}  &  1 &  2  \\ 
 DTLZ1& \cite{Deb2005} & 7 & 3 \\
 DTLZ2 & \cite{Deb2005} & 12 & 3 \\
 DTLZ3 & \cite{Deb2005} & 12 & 3 \\
 DTLZ4 & \cite{Deb2005} & 12 & 3 \\
 DTLZ5 & \cite{Deb2005} & 14 & 5 \\
 FA1&  \cite{reviewproblems}  &  3 &  3 \\ 
 Far1&  \cite{reviewproblems}  &  2 &  2 \\ 
 FDS &  \cite{fliegeanddrummond2009}  & 5 & 3  \\  
 FF1 & \cite{reviewproblems} &     2 &  2 \\ 
 Hil1&  \cite{hill} &     2 &  2 \\ 
 IKK1&  \cite{reviewproblems}  &  2 &  3 \\ 
 IM1&  \cite{reviewproblems}  &  2 &  2 \\ 
 JOS1 & \cite{10.5555/2955239.2955427} & 50 & 2 \\
 JOS4 & \cite{10.5555/2955239.2955427} & 50 & 2 \\
 KW2 & \cite{kw2} &  2 &  2 \\ 
 LE1&  \cite{reviewproblems}  &  2 &  2 \\ 
 Lov1& \cite{doi:10.1137/100784746} & 2 & 2 \\
 Lov2& \cite{doi:10.1137/100784746} & 2 & 2 \\
 Lov3& \cite{doi:10.1137/100784746} & 2 & 2 \\
 Lov4& \cite{doi:10.1137/100784746} & 2 & 2 \\
 Lov5& \cite{doi:10.1137/100784746} & 3 & 2 \\
 Lov6& \cite{doi:10.1137/100784746} & 6 & 2 \\
 LTDZ1 & \cite{laumanns2002combining} & 3 & 3 \\
 MGH9& \cite{moretest}  &   3 &  15 \\  
 MGH16& \cite{moretest}  &   4 &  5  \\  
 MGH26 & \cite{moretest} &    4 &  4 \\ 
 MGH33 & \cite{moretest} &    10 &  10 \\ 
\hline
\end{tabular}
&
\rowcolors{2}{}{lightgray}
\begin{tabular}{|l|l|c|c|}
\hline
\rowcolor[gray]{.90}
Problem & Source & $n$ & $m$  \\
\hline
MHHM1 & \cite{reviewproblems} & 1 & 3 \\
 MHHM2&  \cite{reviewproblems}  &  2 &  3  \\ 
 MLF1&  \cite{reviewproblems}  &    1 &  2  \\ 
MLF2&  \cite{reviewproblems}  &    2 &  2 \\
 MMR1  & \cite{italianos} &     2 &  2 \\ 
 MMR2  & \cite{italianos} &     2 &  2 \\ 
 MMR3  & \cite{italianos} &     2 &  2  \\ 
 MMR4  & \cite{italianos} &     3 &  2  \\ 
 MOP2 & \cite{reviewproblems} & 3 & 2 \\
 MOP3 & \cite{reviewproblems} &     2 &  2  \\ 
 MOP5&  \cite{reviewproblems} &     2 &  3 \\ 
 MOP6&  \cite{reviewproblems} &     2 &  2 \\ 
 MOP7&  \cite{reviewproblems} &     2 &  3 \\ 
 PNR&  \cite{pnr} &  2 &  2 \\ 
 QV1 & \cite{reviewproblems} & 16 & 2 \\
 SD&  \cite{doi:10.2514/5.9781600866234.0209.0249} &  4 &  2 \\ 
 SK1 & \cite{reviewproblems} &      1 &  2 \\ 
 SK2 & \cite{reviewproblems} &      4 &  2 \\ 
 SLCDT1&  \cite{slcdt} &     2 &  2 \\ 
 SLCDT2&  \cite{slcdt} &   10 &  3  \\ 
 SP1&  \cite{reviewproblems} &     2 &  2  \\ 
 SSFYY2&  \cite{reviewproblems} &     1 &  2 \\ 
TKLY1&  \cite{reviewproblems} &     4 &  2  \\ 
 Toi4&  \cite{tointtest} &  4 &  2 \\ 
 Toi8 & \cite{tointtest} & 3 & 3 \\
 Toi9 & \cite{tointtest}  & 4   & 4   \\ 
 Toi10 &\cite{tointtest}  & 4 & 3      \\  
 VU1&  \cite{reviewproblems} &   2 &  2  \\ 
 VU2&  \cite{reviewproblems} &   2 &  2  \\ 
 ZDT1 & \cite{ZDT} & 30 & 2 \\ 
 ZDT2 & \cite{ZDT} & 30 & 2 \\ 
 ZDT3 & \cite{ZDT} & 30 & 2 \\ 
 ZDT4 & \cite{ZDT} & 10 & 2 \\
 ZDT6 & \cite{ZDT} & 10 & 2  \\ 
 ZLT1 & \cite{reviewproblems} & 100 & 2 \\
\hline
\end{tabular}
\end{tabular}
\end{table}

The numerical results are presented by means of performance profile plots \cite{Dolan2002}, which are a standard tool for comparing several methods on a large collection of test problems. To compare the behavior of each algorithm, we consider all $56{,}800$ instances arising from $71$ problems, $4$ uncertainty levels, and $200$ starting points; the profiles include the $49{,}699$ instances solved by at least one method ($49{,}675$ for the execution time, which excludes the runs whose time includes compilation). Figures~\ref{fig:pp-nf-nj_evals} and \ref{fig:pp-iter-time} display performance profiles for PDFPM, ProxGrad, and CondG with respect to the number of function evaluations, the number of gradient evaluations, the iteration count, and the total execution time. 
In function evaluations, PDFPM is the most efficient method: it has the lowest cost on $55\%$ of the instances, against $38\%$ for CondG and $15\%$ for ProxGrad, and it remains ahead for performance ratios up to about $2^5$. In gradient evaluations and iterations, CondG is the most efficient method, with the lowest cost on $58\%$ of the instances, against $34\%$ for ProxGrad and $33\%$ for PDFPM; in execution time, ProxGrad is the most efficient, with $50\%$, against $25\%$ for PDFPM and for CondG. ProxGrad solves the largest fraction of the instances ($90.3\%$), followed by PDFPM ($86.4\%$) and CondG ($85.9\%$).

The profiles above use exact gradients and measure the cost of reaching a point with $w_1\le 10^{-6}$. To assess the effect of approximating the gradients, we also ran PDFPM with central finite differences, as described at the beginning of Subsection~\ref{sub_sec_numerical}, on the same instances, with the same matrices $\Tilde{A}_j$ and starting points; $w_1$ is still evaluated with the exact gradients, since it depends only on the problem. The two versions solve practically the same instances: of the $42{,}959$ instances solved by at least one of them, $42{,}939$ with exact gradients and $42{,}933$ with finite differences, both of them $42{,}913$, in practically the same number of iterations; the finite differences cost $2n$ additional function evaluations per approximation of the gradients. We therefore compare the two versions by the quality of the set of points obtained from the different starting points, using four indicators of the approximation of the Pareto front: Purity \cite{Bandyopadhyay2004Multiobjective}, the fraction of the points of a version that are not dominated by the points of the other; the $\Gamma$-spread and the $\Delta$-spread \cite{doi:10.1137/10079731X,Deb2002Fast}, which measure the largest gap and the uniformity of the distribution of the points along the front; and the hypervolume \cite{Zitzler1998Multiobjective}, the size of the region of the objective space dominated by the front, computed for the instances with at most five objectives. For each pair (problem, $\delta$), the front of a version consists of the nondominated objective vectors at the iterates where its successful runs attain $w_1\le 10^{-6}$, and the reference front consists of the nondominated points of the union of the two fronts. Figures~\ref{fig:pp-purity-hv} and~\ref{fig:pp-spread} show the performance profiles of these indicators on the $284$ pairs (problem, $\delta$), built as in \cite{doi:10.1137/10079731X} for Purity and the spread indicators and with cost $1/\text{HV}$ for the hypervolume. The two versions are practically tied in the four indicators: the fraction of the pairs on which each version attains the best value, among the pairs where the indicator can be compared ($274$, $253$, $251$ and $246$ of the $284$), is $84\%$ and $83\%$ in Purity, $49\%$ and $51\%$ in the $\Gamma$-spread, $50\%$ and $51\%$ in the $\Delta$-spread, and $48\%$ and $52\%$ in the hypervolume, for exact gradients and finite differences, respectively. Hence, on these problems, approximating the gradients by central finite differences increases the number of function evaluations but does not deteriorate the approximation of the Pareto front.

\begin{figure}[h]
\centering
\begin{minipage}{0.48\textwidth}
\centering
\includegraphics[width=\textwidth]{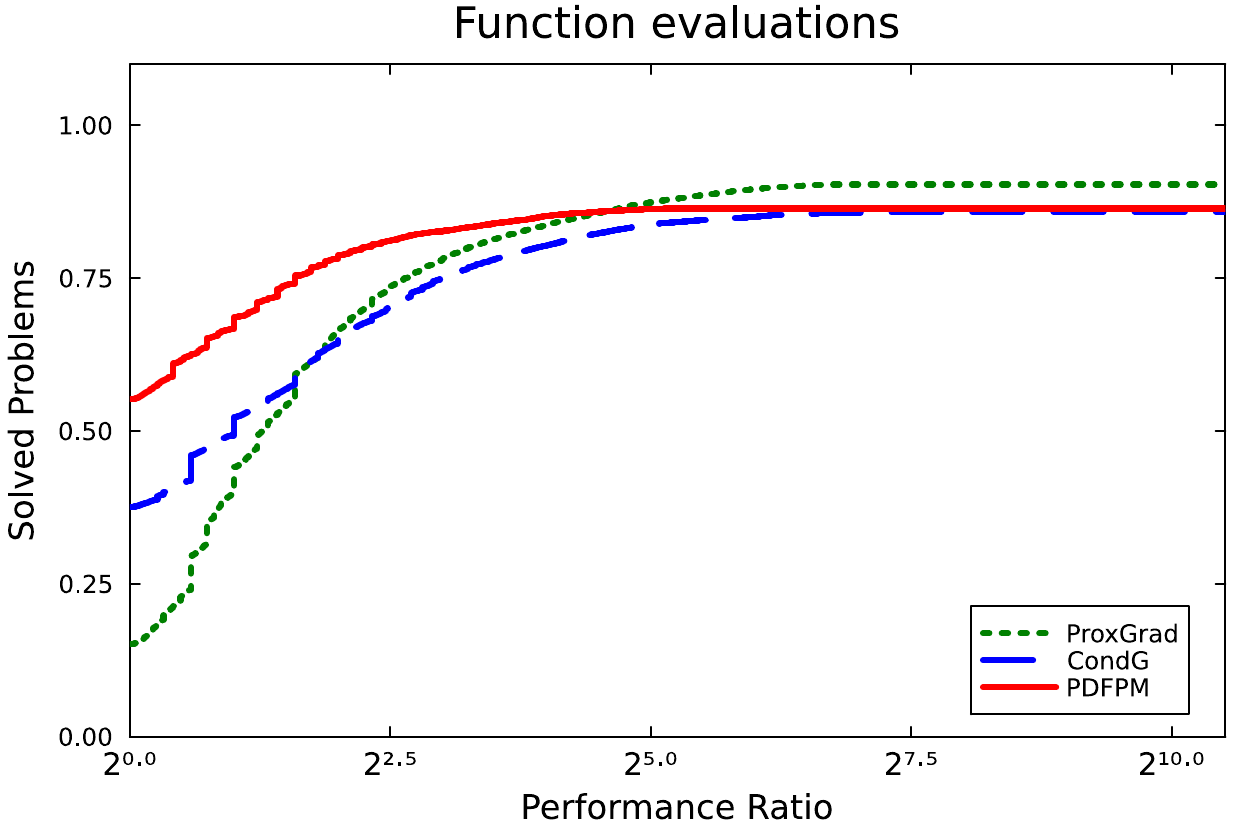}
\end{minipage}
\hfill
\begin{minipage}{0.48\textwidth}
\centering
\includegraphics[width=\textwidth]{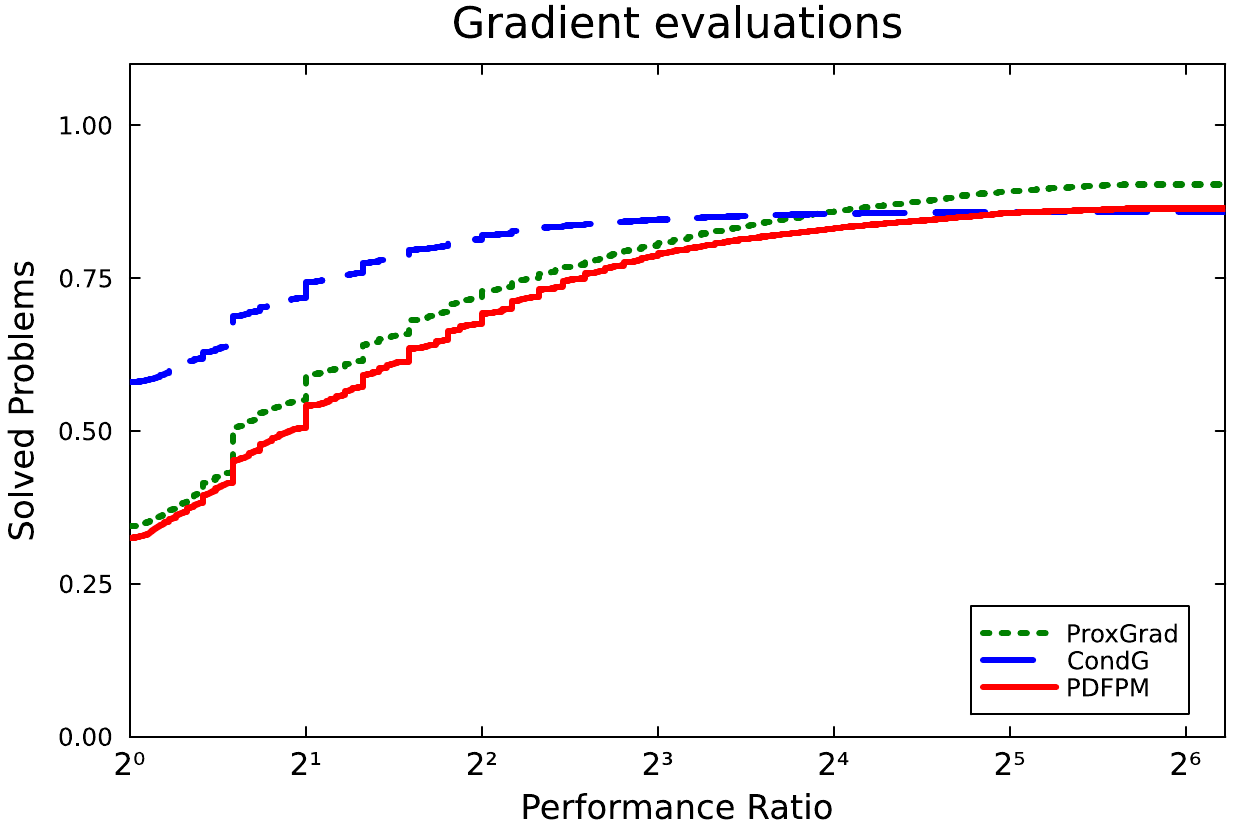}
\end{minipage}
\caption{Performance profiles with respect to the number of function and gradient evaluations  of PDFPM with exact gradients and with BFGS matrices $B_j^k$, CondG, and ProxGrad on the instances of the $71$ problems in Table~\ref{tab:problems}; a method solves an instance at its first iterate with $w_1\le 10^{-6}$.}
\label{fig:pp-nf-nj_evals}
\end{figure}

\begin{figure}[h]
\centering
\begin{minipage}{0.48\textwidth}
\centering
\includegraphics[width=\textwidth]{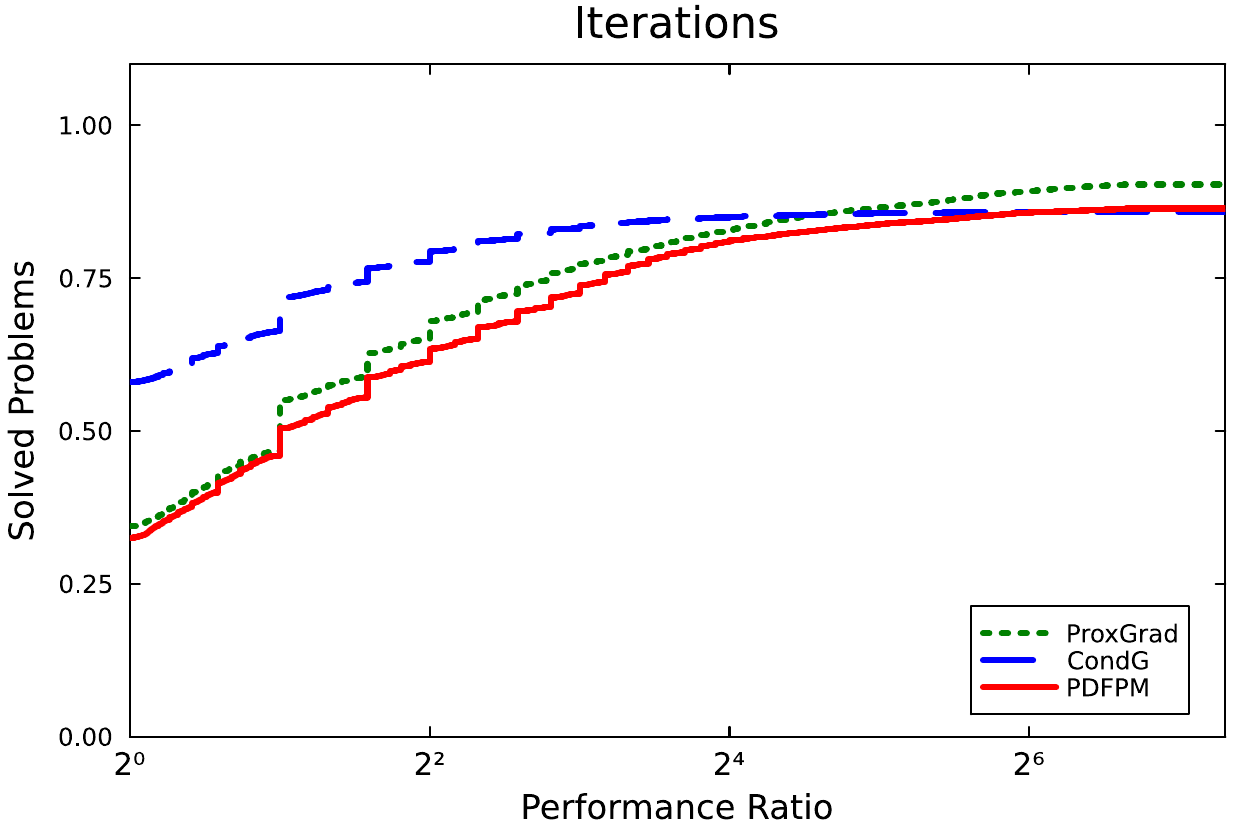}
\end{minipage}
\hfill
\begin{minipage}{0.48\textwidth}
\centering
\includegraphics[width=\textwidth]{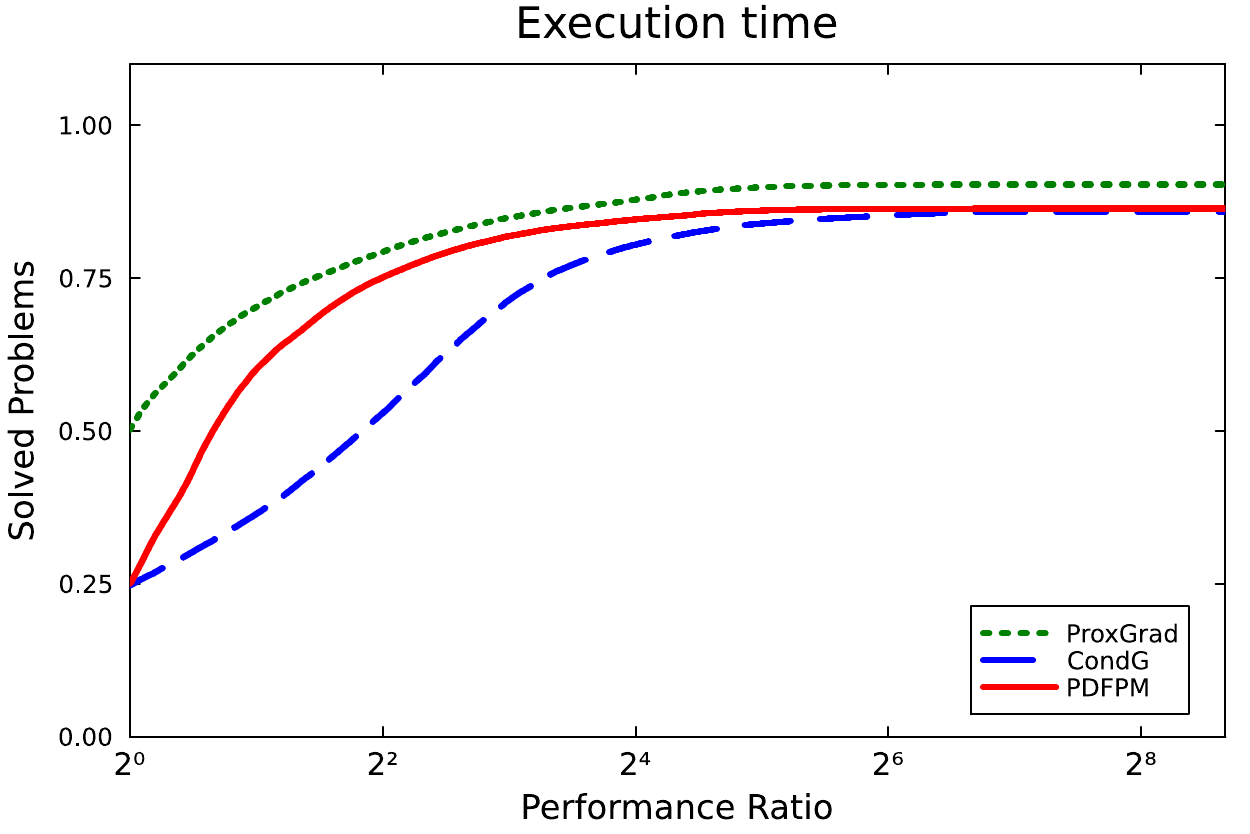}
\end{minipage}
\caption{Performance profiles with respect to the number of iterations and the total execution time; methods and instances as in Figure~\ref{fig:pp-nf-nj_evals}.}
\label{fig:pp-iter-time}
\end{figure}

\begin{figure}[!htbp]
\centering
\begin{minipage}{0.48\textwidth}
\centering
\includegraphics[width=\textwidth]{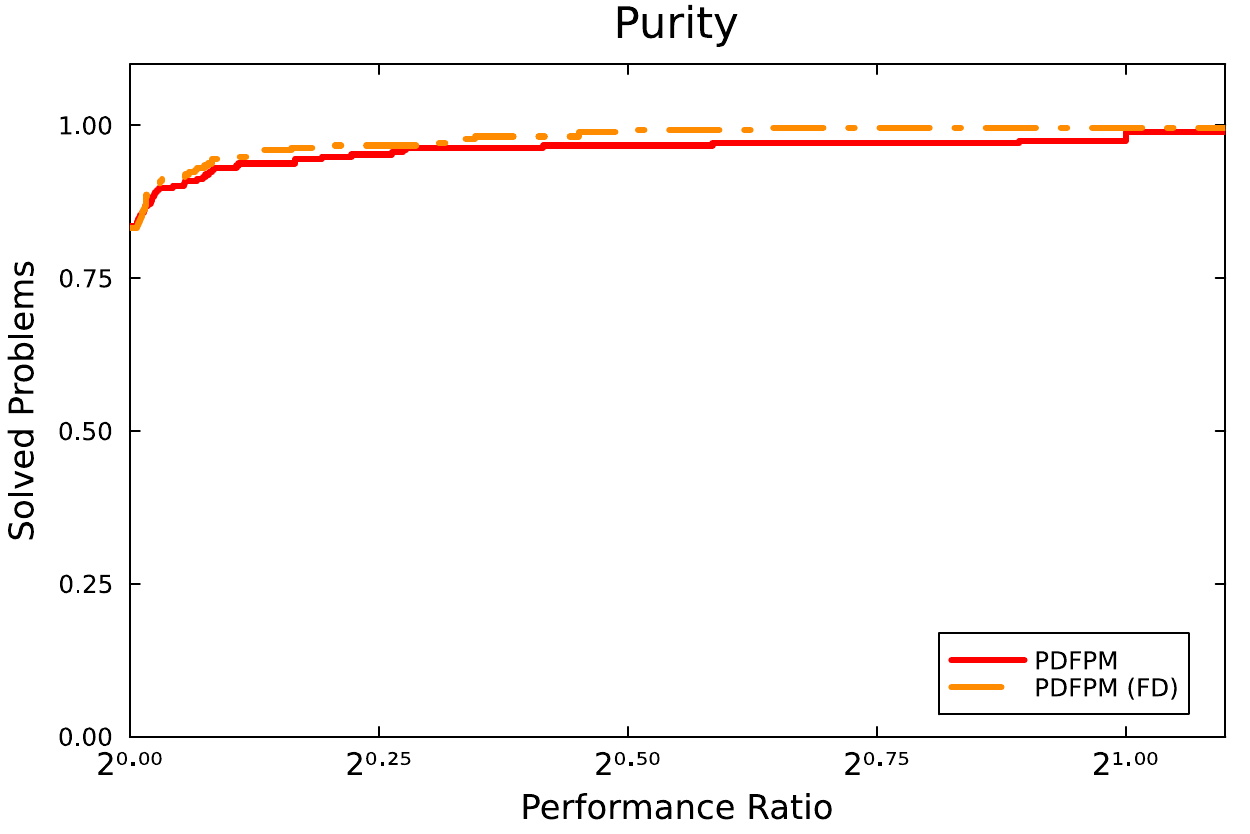}
\end{minipage}
\hfill
\begin{minipage}{0.48\textwidth}
\centering
\includegraphics[width=\textwidth]{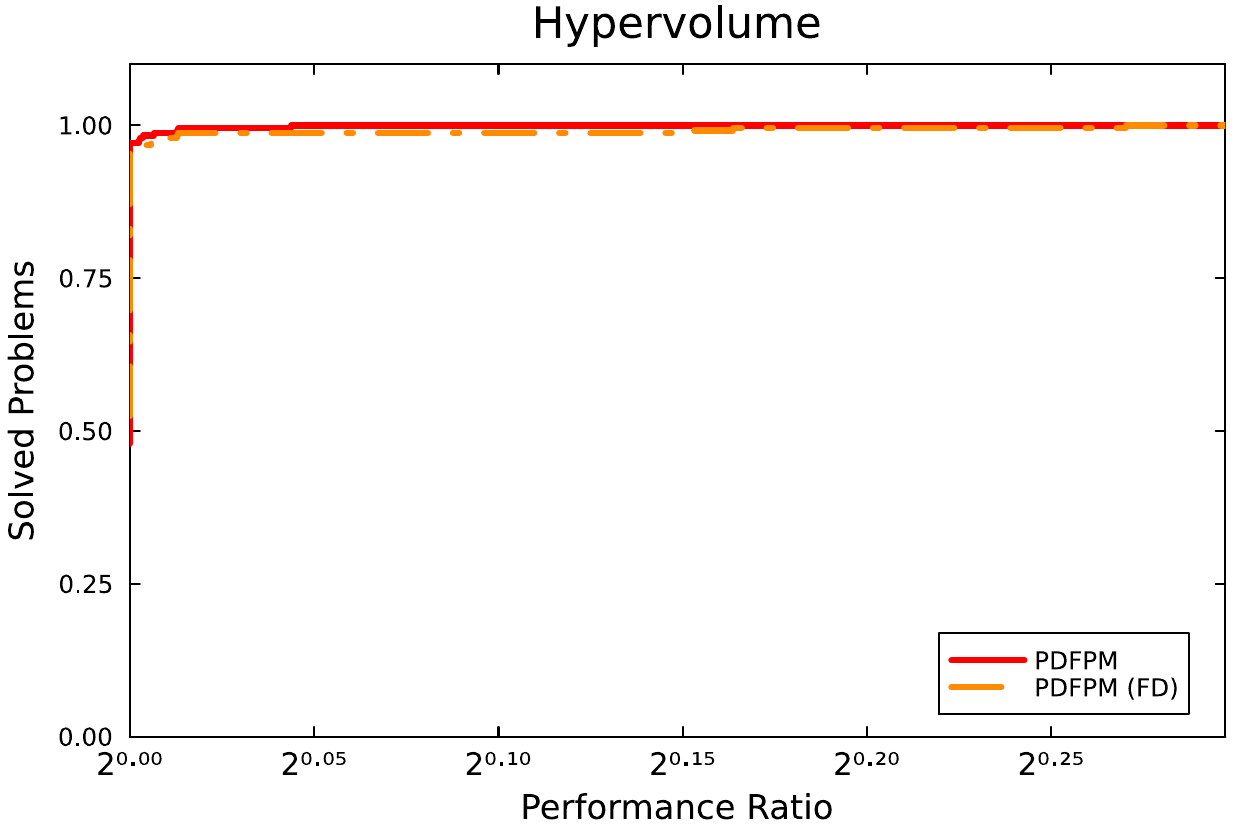}
\end{minipage}
\caption{Performance profiles of Purity and hypervolume of PDFPM with exact gradients and with central finite differences.}
\label{fig:pp-purity-hv}
\end{figure}

\begin{figure}[!htbp]
\centering
\begin{minipage}{0.48\textwidth}
\centering
\includegraphics[width=\textwidth]{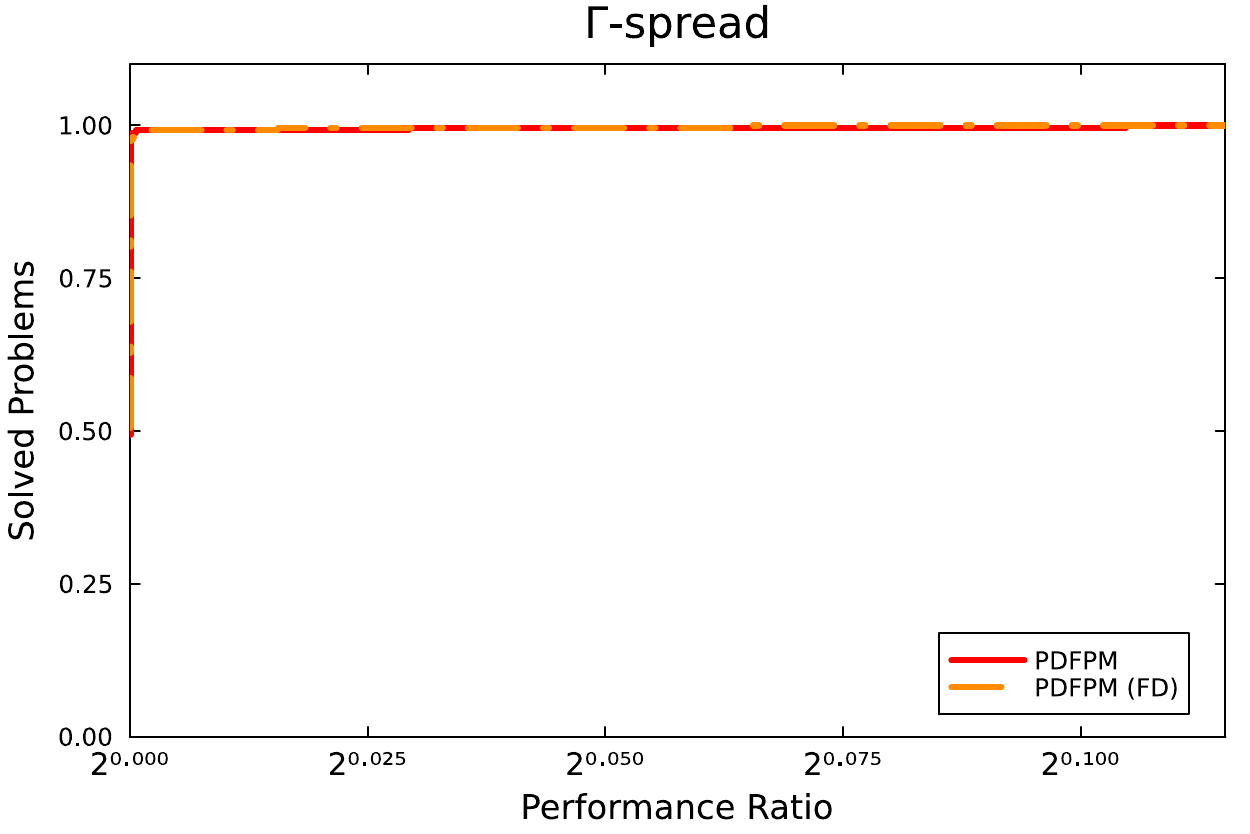}
\end{minipage}
\hfill
\begin{minipage}{0.48\textwidth}
\centering
\includegraphics[width=\textwidth]{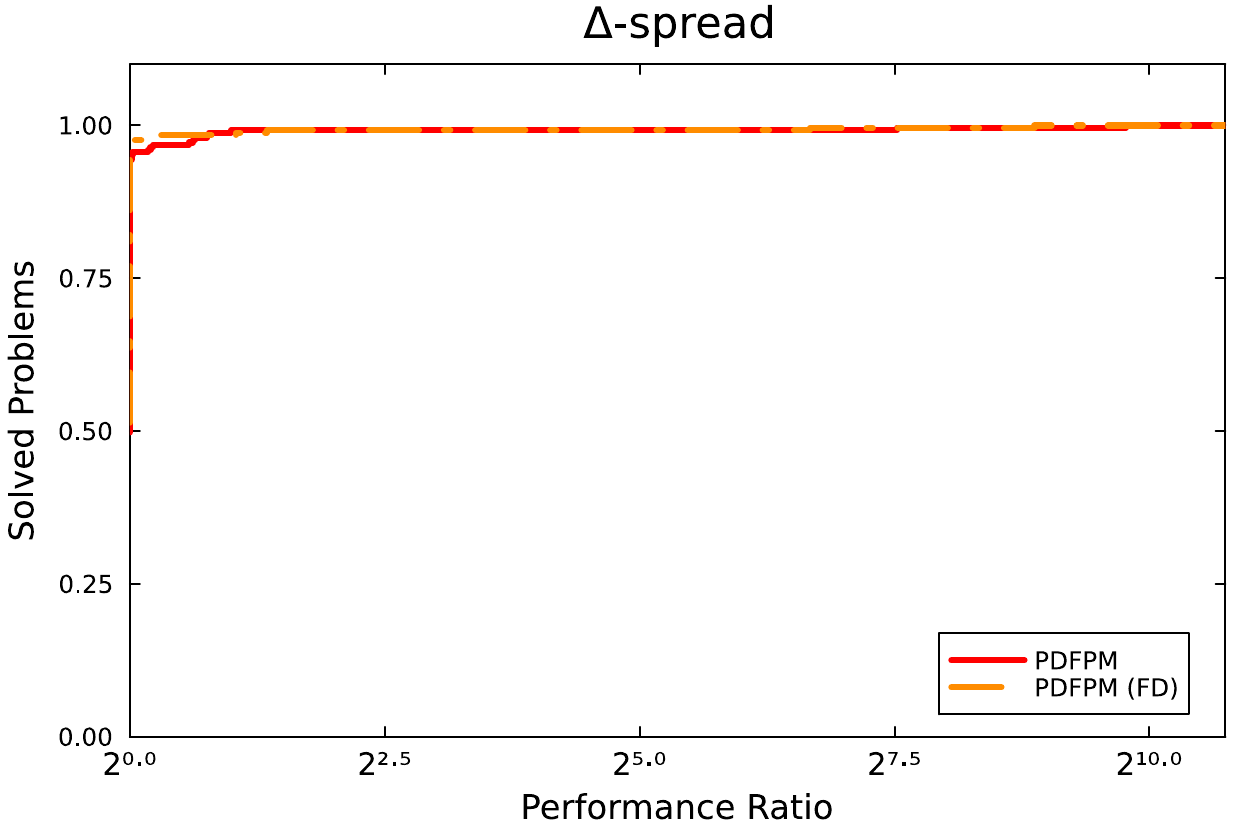}
\end{minipage}
\caption{Performance profiles of the $\Gamma$-spread and $\Delta$-spread of PDFPM with exact gradients and with central finite differences.}
\label{fig:pp-spread}
\end{figure}

\clearpage
\section{Conclusions}\label{conlusions}
In this work, we addressed a multiobjective optimization problem involving the composition of a differentiable function with a convex one, under the assumption that the gradient of each differentiable component can be decomposed into parts with Lipschitz and H\"{o}lder regularity. This flexible modeling covers both classical cases and mixed scenarios with different regularities across objectives. One of the main contributions of the method is its flexibility in handling approximations of the differentiable part, allowing exact derivatives, first-order approximations (such as finite differences), or second-order information through a matrix $B_j^k$, which may be either the exact Hessian or an approximation obtained via methods such as BFGS, secant Newton, or finite differences, provided that $\{B_k\}$ is a bounded sequence of positive semidefinite matrices. This makes the method applicable to a wide range of problems, including situations where derivative information is only partially or not directly available. Moreover, large values of the regularization parameter $\sigma$ may lead to small steps and slow convergence, so it is advisable to start with a small $\sigma$. The method does not require prior knowledge of the Lipschitz–H\"{o}lder constants. Moreover, under suitable regularity and parameter conditions, its complexity analysis guarantees convergence to an $\epsilon$-Pareto critical point in at most $\mathcal{O}\left(\epsilon^{-\frac{\beta+1}{\beta}}\right)$ iterations, where $\beta=\min\{\beta_j:j=1,\ldots,m\}$ and $\beta_j$ denotes the Hölder exponent associated with the gradient of $f_j$. This bound is consistent with the best-known complexity results for this type of regularity.

The numerical experiments focused on robust multiobjective optimization problems. In the biobjective examples AAS1 and AAS2, we observed that, as the level of uncertainty increases, the Pareto fronts approximated by PDFPM move away from the nominal front corresponding to the case without uncertainty, as expected from the robust formulation. At the same time, about $97\%$ of the points returned by PDFPM were weakly Pareto optimal up to $u_0\le 10^{-8}$ on both problems, which suggests that the method handles uncertainty in a stable way on these problems. In a larger benchmark study involving standard multiobjective test problems and several uncertainty levels, a performance-profile comparison with two first-order composite methods, in which the three methods were measured at the same accuracy by a common merit function, indicated that PDFPM requires the fewest function evaluations, whereas ProxGrad solves the largest fraction of the instances and is the fastest, and CondG requires the fewest iterations. With the gradients approximated by central finite differences, PDFPM reached this accuracy on practically the same runs and produced approximations of the Pareto fronts of the same quality, at the cost of additional function evaluations.  

We developed a partially derivative-free method, and for future work, it would be interesting to consider a fully derivative-free algorithm. Additionally, other situations could be explored, such as vector environments and closed pointed convex cones.

\bmhead{Acknowledgements} This work has been partially supported by CNPq grant 407147/2023-3 (V. S. Amaral).

\bmhead{Data availability} {The code that generates the data of Section~\ref{sec5} is available at \url{https://github.com/Souza-DR/AAS2025-PDFreeMO}. The data generated during the current study are available from the authors on reasonable request.}
\subsection*{Declarations}
The authors have no conflicts of interest to declare that are relevant to the content of 
this article.

\bibliography{sn-bibliography}
\end{document}